\documentclass[11pt]{amsart}
\usepackage[utf8]{inputenc}
\usepackage[margin=1in,letterpaper,portrait]{geometry}
\usepackage{ytableau}
\ytableausetup{notabloids}
\ytableausetup{centertableaux}
\usepackage{amsfonts}
\usepackage{amsbsy,amsthm,mathtools}
\usepackage{latexsym,amsmath,amssymb,verbatim,tikz}
\usepackage{graphicx}
\usepackage{hyperref}

\theoremstyle{plain}
   
   \newtheorem{theorem}{Theorem}[section]
   \newtheorem{proposition}[theorem]{Proposition}
   
   \newtheorem{lemma}[theorem]{Lemma}
   \newtheorem{corollary}[theorem]{Corollary}

\theoremstyle{definition}
   \newtheorem{definition}[theorem]{Definition}
   \newtheorem{defn}[theorem]{Definition}

   \newtheorem{remark}[theorem]{Remark}

   \newcommand{\f}{\mathcal F}

\newcommand{\ch}{{\operatorname{ch}}}

\numberwithin{equation}{section}

\DeclareMathOperator\sign{sign}

\newcommand{\ve}{\varepsilon}
\newcommand{\Des}{{\operatorname{Des}}}
\newcommand{\C}{C}
\newcommand{\roots}{\rho}
\newcommand{\sroots}{{\overline{\rho}}}
 
\newcommand{\up}{{\operatorname{\uparrow}}}

\def\NN{{\mathbb N}}
\def\ZZ{{\mathbb Z}}
\def\CC{{\mathbb C}}

\newcommand{\Q}{\mathcal Q}

\newcommand{\blambda}{{\underline\lambda}}

\newcommand{\bnu}{{\underline\nu}}

\newcommand{\bs}{{\underline s}}
\newcommand{\bt}{{\underline t}}

\newcommand{\oomega}{{\overline\omega}}

\newcommand{\OP}{OP}
\newcommand{\EP}{EP}

\title[Descent sets for permutations with cycles of odd or even lengths]{Descent set distribution for permutations \\ with cycles of only odd or only even lengths}
\author{Ron M.\ Adin}
\address{Department of Mathematics, Bar-Ilan University, 
	Ramat-Gan 52900, Israel}
\email{radin@math.biu.ac.il}

\author{P\'al Heged\H{u}s}
\address{Department of Algebra and Geometry, Institute of Mathematics, Budapest University of Technology and Economics, M\H uegyetem rkp. 3., H-1111 Budapest, Hungary}
\email{hegpal@math.bme.hu}

\author{Yuval Roichman}
\address{Department of Mathematics, Bar-Ilan University, 
Ramat-Gan 52900, Israel}
\email{yuvalr@math.biu.ac.il}

\date{February 5, 2025}

\begin{document}
 
\begin{abstract}
    It is known that 
    the number of permutations in the symmetric group $S_{2n}$ with cycles of odd lengths only is equal to the number of permutations with cycles of even lengths only.
    We prove a refinement of this equality, involving descent sets: the number of permutations in $S_{2n}$ with a prescribed descent set and all cycles of odd lengths is equal to the number of permutations with the complementary descent set and all cycles of even lengths.  
    There is also a variant for $S_{2n+1}$. 
    The proof uses generating functions for character values and applies a new identity on higher Lie characters.  
\end{abstract}

\maketitle

\tableofcontents

\section{Introduction}

For a positive even integer $n$, let $OC(n)$ be the set of all permutations $\pi \in S_n$ with cycles of odd lengths only,   
and let $EC(n)$ be the set of all permutations $\pi \in S_n$ with cycles of even lengths only. 
It is known that, for every positive even $n$, 
\[
    |OC(n)| = |EC(n)| =  (n-1)!!^2, 
\]
where the double factorial $(n-1)!! := (n-1) \cdot (n-3) \cdot (n-5) \cdots$.
See, e.g.,~\cite[Theorem 6.24]{Bona} and [OEIS A001818]. 

In this paper we prove the following remarkable refinement.

First, let us extend the above definitions to odd values of $n$.
For a positive odd integer $n$, let $OC(n)$ be the set of all permutations $\pi \in S_n$ with cycles of odd lengths only,   
and let $EC(n)$ be the set of all permutations $\pi \in S_n$ with cycles of even lengths only, plus a single fixed point. 
The {\em descent set} of a permutation $\pi = [\pi_1, \ldots, \pi_n] \in S_n$ 
is
\[
    \Des(\pi) := \{1 \le i \le n-1 \,:\, \pi_i > \pi_{i+1} \}
    \quad \subseteq [n-1],
\]
where $[m]:=\{1,2,\ldots,m\}$.

\begin{theorem}\label{t:equid}
    For any positive integer $n$ and subset $J \subseteq [n-1]$,
    \[
        |\{\pi \in OC(n) \,:\, \Des(\pi) = J\}| 
        = |\{\pi \in EC(n) \,:\, \Des(\pi) = [n-1] \setminus J\}|.
    \]
\end{theorem}


Theorem~\ref{t:equid} is equivalent to an apparently new identity on higher Lie characters, Theorem~\ref{t:OP-EP}. 
For any partition $\lambda \vdash n$, let $\psi^\lambda_{S_n}$ be the corresponding higher Lie character of $S_n$ (for a definition see Subsection~\ref{subsec:def_hLc_A}). 
For any positive integer $n$, let $\OP(n)$ be the set of all partitions of $n$ with odd parts only.
For a positive even integer $n$, let $\EP(n)$ be the set of all partitions of $n$ with even parts only;
and, for a positive odd integer $n$, let $\EP(n)$ be the set of all partitions of $n$ with all parts even, except for a single part of size $1$.

\begin{theorem}\label{t:OP-EP}
    For any positive integer $n$,
    \[
        \sum_{\lambda \in \OP(n)} \psi_{S_n}^\lambda
        = \sign \otimes \sum_{\lambda \in \EP(n)} \psi_{S_n}^\lambda ,
    \]
    where $\sign$ is the sign character of $S_{n}$.
\end{theorem}


This result follows, in turn, from two explicit generating functions, Theorem~\ref{t:odd_cycles} and Theorem~\ref{t:even_cycles}.
For a partition $\nu \vdash n$, let $b_{j}$ be the number of parts of size $j$ in $\nu$ $(\forall\, j \ge 1)$. 
Let
\[
    |Z_\nu| = \frac{n!}{\prod_j b_j! j^{b_j}}
\]
be the size of the centralizer $Z_\nu$ of any element of cycle type $\nu$ in $S_n$. 
Let $\bt = (t_{j})_{j \ge 1}$ be a countable set of indeterminates, 
consider the ring $\CC[[\bt]]$ of formal power series in these indeterminates, 
and denote 
$\bt^{c(\nu)} := \prod_{j} t_{j}^{b_{j}}$.

\begin{theorem}\label{t:odd_cycles}
    \[
        \sum_{n \ge 0}  
        \sum_{\lambda \in \OP(n)} 
        \sum_{\nu \vdash n} 
        \psi_{S_n}^\lambda(\nu) \,
        \frac{\bt^{c(\nu)}}{|Z_\nu|} 
        = \prod_{p \ge 0} 
        \left( \frac{1 + t_{2^p}}{1 - t_{2^p}} \right)^{1/2^{p+1}} .
    \]
\end{theorem}

\begin{theorem}\label{t:even_cycles}
    \[
        \sum_{n \ge 0}  
        \sum_{\lambda \in \EP(n)} 
        \sum_{\nu \vdash n} 
        \sign(\nu) \psi_{S_n}^\lambda(\nu) \,
        \frac{\bt^{c(\nu)}}{|Z_\nu|} 
        = \prod_{p \ge 0} 
        \left( \frac{1 + t_{2^p}}{1 - t_{2^p}} \right)^{1/2^{p+1}} .
    \]
\end{theorem}

In particular, setting $t_{2^p} = 0$ for all $p \ge 1$, we obtain 
the following result, to be compared with [OEIS A000246].

\begin{corollary}\label{cor:double}
    For any integer $n \ge 2$,
    \[
        \sum_{\lambda \in \OP(n)} 
        \dim \psi_{S_n}^\lambda 
        = \sum_{\lambda \in \EP(n)} 
        \dim \psi_{S_n}^\lambda 
        = \begin{cases}
            (n-1)!!^2, &\text{if $n$ is even;} \\
            n!! \cdot (n-2)!!, &\text{ if $n$ is odd.}
        \end{cases}    
    \]
\end{corollary}


The group of signed permutations $B_n$ can be viewed as the centralizer, in $S_{2n}$, of a fixed-point-free involution (a permutation of cycle type $(2, \ldots, 2)$). 
Noting that its index $|S_{2n}|/|B_n| = (2n-1)!!$, Corollary~\ref{cor:double} suggests that the sum of higher Lie characters of $S_{2n}$ over the partitions in $\OP(2n)$ is induced from a character of $B_n$. We prove that this is, indeed, the case (with an analogue for $S_{2n+1}$). 
Specifically, denote
\[
    \eta_{B_n}
    := \sum_{\lambda \vdash n} \psi_{B_n}^{(\lambda, \varnothing)} ,
\]
the sum of (type $B$) higher Lie characters of $B_n$ corresponding to conjugacy classes with positive cycles only (for a definition see Subsection~\ref{subsec:def_hLc_B}).  

\begin{theorem}\label{t:induced_to_even_main}
    For any integer $n \ge 0$,
    \[
        \sum_{\lambda \in \OP(n)}         \psi_{S_n}^\lambda
        = \eta_{B_{\lfloor n/2 \rfloor}} \up_{B_{\lfloor n/2 \rfloor}}^{S_{n}}.
    \]
\end{theorem}

\noindent 
See Theorem~\ref{t:induced_to_even} below.


\medskip

The rest of the paper is organized as follows.
Higher Lie characters are defined in Subsection~\ref{subsec:def_hLc_A},
and the equivalence of Theorems~\ref{t:equid} and~\ref{t:OP-EP} is explained in Subsection~\ref{subsec:quasi}. 
In Section~\ref{sec:gf_hLc_A} we 
compute a generating function (Theorem~\ref{t:summation_4_A}) for the values of the higher Lie characters of $S_n$. 
In Section~\ref{sec:gf_hLc_A_signed} we state a signed analogue, Theorem~\ref{t:summation_4_A_signed}, and outline its proof.
In Section~\ref{sec:proofs} we deduce Theorem~\ref{t:odd_cycles} (from Theorem~\ref{t:summation_4_A}) and Theorem~\ref{t:even_cycles} (from Theorem~\ref{t:summation_4_A_signed}), thus proving Theorem~\ref{t:OP-EP}.
In Section~\ref{sec:roots} we study related character identities for odd root enumerators.
Finally, in Section~\ref{sec:induction} we 
prove Theorem~\ref{t:induced_to_even_main}.

\bigskip

\noindent
{\bf Acknowledgements.} The authors thank Doron Zeilberger for stimulating discussions,   
and Stoyan Dimitrov and Sergi Elizalde for useful comments and references.

\section{Preliminaries}\label{sec:prel}

In this section we define higher Lie characters (in $S_n$), and use basic properties of quasisymmetric functions to explain the equivalence of Theorems~\ref{t:equid} and~\ref{t:OP-EP}.

\subsection{Higher Lie characters}
\label{subsec:def_hLc_A}

In this subsection we define a higher Lie character for any element (in fact, for any conjugacy class) in the symmetric group $S_n$. 
Let us first recall the well-known description of the centralizer of an arbitrary element of $S_n$.

We write $\lambda\vdash n$ to denote that $\lambda$ is a partition of a positive integer $n$.

\begin{lemma}\label{t:centralizer_in_Sn}
    {\rm (Centralizers in $S_n$)}
%
        Let $x \in S_n$ be an element of cycle type $\lambda\vdash n$ where, for each $i \ge 1$, the partition $\lambda$ has $a_{i}$ parts of size $i$. 
        Write 
        \[
            x = \prod_{i \ge 1} x_{i}\,, 
        \]
        where $x_{i}$ has $a_{i}$ cycles of length $i$ $(i \ge 1)$; of course, only finitely many factors here are nontrivial.
        Then the centralizer of $x$ in $S_n$,    $Z_x = Z_{S_n}(x)$, satisfies
        \[
            Z_{S_n}(x) = \bigtimes_{i \ge 1} Z_{S_{ia_{i}}}(x_{i}) 
            \cong \bigtimes_{i \ge 1} G_{i} \wr S_{a_{i}} ,
        \]
        where $G_{i}$, isomorphic to the cyclic group of order $i$, is the centralizer in $S_i$ of a cycle of length $i$.
        By convention, $G_i \wr S_0$ is the trivial group while $G_i \wr S_1 \cong G_i$. 
%
%
\end{lemma}

\begin{remark}
    We use $G_i$, rather than $C_i$ or $\ZZ_i$, to denote the cyclic group of order $n$,
    since $C$ and $Z$ are intensively used here for other purposes.
\end{remark}

\begin{definition}\label{def:HLC_A}
    {\rm (Higher Lie characters in $S_n$)}
    Let $x$ be an element of cycle type $\lambda$ in $S_n$, as in Lemma~\ref{t:centralizer_in_Sn}.        
    \begin{enumerate}
    
        \item[(a)]
        For each $i \ge 1$, let $\omega_{i}$ be the linear character on $G_{i} \wr S_{a_{i}}$ 
        which is equal to a primitive irreducible character on the cyclic group $G_{i}$,
        and trivial on the wreathing group $S_{a_{i}}$.
        Let
        \[
            \omega^x 
            := 
            \bigotimes_{i \ge 1} \omega_{i} ,
        \]
        a linear character on $Z_x$. 
        
        \item[(b)]
        Define the corresponding {\em higher Lie character} to be the induced character
        \[
            \psi_{S_n}^x 
            := \omega^x 
            \up_{Z_x}^{S_n}. 
        \]

        \item[(c)]
        It is easy to see that $\psi_{S_n}^x$ depends only on the conjugacy class $C$ (equivalently, the cycle type $\lambda$) of $x$, and can therefore be denoted $\psi_{S_n}^C$ or $\psi_{S_n}^\lambda$.
        
    \end{enumerate}
\end{definition}

\subsection{Quasisymmetric functions and descents}\label{subsec:quasi} 

The {\em descent set} of a permutation $\pi = [\pi_1, \ldots, \pi_n]$ in the symmetric group $S_n$ on $n$ letters is
\[
    \Des(\pi) := \{1 \le i \le n-1 \,:\, \pi_i > \pi_{i+1} \}
    \quad \subseteq [n-1],
\]
where $[m]:=\{1,2,\ldots,m\}$. 
	
\begin{defn}\label{def:fundamental_QSF}
    For each subset $D \subseteq [n-1]$ define the {\em fundamental quasisymmetric function}
    \[
	\f_{n,D}({\bf x}) := \sum_{\substack{i_1\le i_2 \le \ldots \le i_n \\ {i_j < i_{j+1} \text{ if } j \in D}}} 
	x_{i_1} x_{i_2} \cdots x_{i_n}.
    \]
\end{defn}

Given any subset $A\subseteq S_n$, define the quasisymmetric function
\[
    \Q(A) := \sum_{\pi \in A} \f_{n,\Des(\pi)}.
\]
%
%
%
%
%
%
%
%
%
In their seminal paper~\cite{GR}, Gessel and Reutenauer prove the following. 


\begin{theorem}\label{t:GR22}\cite[Theorem 3.6]
{GR}
    For every partition $\lambda \vdash n$, let  $C_\lambda$ be the conjugacy class of permutations in $S_n$ of cycle type $\lambda$. Then 
    \[
	\Q(\C_\lambda) = 
	\ch \left( \psi^\lambda_{S_n} \right),
    \]
    where $\ch$ is the Frobenius characteristic map and $\psi^\lambda$ is the higher Lie character 
    defined in Subsection~\ref{subsec:def_hLc_A}.
\end{theorem}

The remarks preceding~\cite[Theorem 4.1]{GR} imply the following variant.

\begin{corollary}
     For every partition $\lambda \vdash n$,
     \[
        \ch \left( \sign \otimes \psi^\lambda_{S_n} \right)
        = \sum_{\pi \in C_\lambda} \f_{n,[n-1] \setminus \Des(\pi)}.
     \]
\end{corollary}

\begin{corollary} 
    Theorem~\ref{t:equid} and Theorem~\ref{t:OP-EP} are equivalent.
\end{corollary}

\begin{proof}
    Recall from~\cite[Ch.~7]{EC2} that the fundamental quasisymmetric functions 
    $\{\f_{n,D} \mid D \subseteq [n-1]\}$ 
    form a basis of the vector space ${\rm{QSym}}_n$ of quasisymmetric functions in $n$ variables. 
    Theorem~\ref{t:equid} is therefore equivalent to
    \[
        \sum_{\pi \in OC(n)} \f_{n,\Des(\pi)}
        = \sum_{\pi \in EC(n)} \f_{n,[n-1] \setminus \Des(\pi)}.  
    \]
    By Theorem~\ref{t:GR22}, this happens if and only if 
    \begin{align*}
        \sum_{\lambda \in \OP(n)} \psi_{S_n}^\lambda 
        &= \ch^{-1} \left( \sum_{\lambda \in \OP(n)} \Q(\C_\lambda) \right)
        = \ch^{-1} \left( \sum_{\pi \in OC(n)} \f_{n,\Des(\pi)} \right) \\
        &= \ch^{-1} \left( \sum_{\pi \in EC(n)} \f_{n,[n-1] \setminus \Des(\pi)} \right)
        = \sign \otimes \sum_{\lambda \in \EP(n)} \psi_{S_n}^\lambda.
        \qedhere
    \end{align*} 
\end{proof}

\section{A generating function for higher Lie characters}
\label{sec:gf_hLc_A}

In this section we state and prove an explicit generating function (Theorem~\ref{t:summation_4_A}) for the values of all the higher Lie characters of the symmetric group $S_n$. 
This formula is an $S_n$-version of a similar formula, recently proved for the higher Lie characters of the hyperoctahedral group $B_n$~\cite[Theorem 4.3]{AHR}. 
%
The main result is stated in Subsection~\ref{subsec:main_hLc_A}, and proved in the following subsections.


\subsection{Main result}
\label{subsec:main_hLc_A}

Let $\lambda$ and $\nu$ be two partitions of $n$. 
For each integer $i \ge 1$, let $a_{i}$ be the number of parts of size $i$ in the partition $\lambda$.
Similarly, for each integer $j \ge 1$, let $b_{j}$ be the number of parts of size $j$ in the partition $\nu$. 
Thus
\[
    \sum_{i} i a_{i} 
    = \sum_{j} j b_{j} 
    = n.
\]
Let $\bs = (s_{i})_{i \ge 1}$ and $\bt = (t_{j})_{j \ge 1}$ be two countable sets of indeterminates.
Consider the ring $\CC[[\bs,\bt]]$ of formal power series in these indeterminates,
and denote 
$\bs^{c(\lambda)} := \prod_{i} s_{i}^{a_{i}}$ 
and $\bt^{c(\nu)} := \prod_{j} t_{j}^{b_{j}}$.
The main result of this section is the following.

\begin{theorem}\label{t:summation_4_A}
    \begin{align*}
        \sum_{n \ge 0}  
        \sum_{\lambda \vdash n} \sum_{\nu \vdash n} 
        \psi_{S_n}^\lambda(\nu) \,
        \frac{\bs^{c(\lambda)} \bt^{c(\nu)}}{|Z_\nu|} 
        &= \exp \left( \sum_{i} \sum_{j} \sum_{e | \gcd(i,j)}
        \mu(e) \,
        \frac{s_{i}^{j/e} t_{j}^{i/e}}{ij/e} 
        \right) ,
    \end{align*}
    where $\mu(\cdot)$ is the classical M\"obius function.
\end{theorem}

We prove this result in the following subsections.

\subsection{Values of induced characters}

Let us start our computations 
by writing a general formula (Lemma~\ref{t:HLC_values}) for the values of each higher Lie character, as an induced character.

For an element $x\in S_n$ denote the conjugacy class and the centralizer of $x$ by $C_x$ and $Z_x$, respectively. 
Recall the definitions of $\omega^x$ and $\psi_{S_n}^x$ from Subsection~\ref{subsec:def_hLc_A}. 

\begin{lemma}\label{t:HLC_values}
    If $x \in S_n$ 
    then
    \[
        \psi_{S_n}^x(g) 
        = \frac{|\C_x|}{|\C_g|} 
        \sum_{z \in \C_g \cap Z_x} 
        \omega^x(z) 
        \qquad (\forall g \in S_n).
    \]
\end{lemma}

\begin{proof}
    Let $G$ be a group, and $\chi$ a character of a subgroup $H$ of $G$. 
    Define a function $\chi^0: G \to \CC$ by
    \[
        \chi^0(g) :=
        \begin{cases}
            \chi(g), &\text{if } g \in H; \\
            0, &\text{if } g \in G \setminus H.
        \end{cases}
    \]
    By~\cite[(5.1)]{Isaacs}, an explicit formula for the induced character $\chi \up_H^G$ is
    \[
        \chi \up_H^G(g)
        = \sum_{G = \cup_a aH} \chi^0(a^{-1} g a)
        = \frac{1}{|H|} \sum_{a\in G} \chi^0(a^{-1} g a)
        \qquad (\forall g \in G) .
    \]
    The mapping $f : G \to \C_g$ defined by
    $f(a) := a^{-1} g a$ $(\forall a \in G)$ is surjective, and satisfies:
    $f(a_1) = f(a_2)$ if and only if $a_1 a_2^{-1} \in Z_g$.
    Hence
    \[
        \chi \up_H^G(g)
        = \frac{|Z_g|}{|H|} \sum_{z \in \C_g}\chi^0(z)
        = \frac{|Z_g|}{|H|} \sum_{z \in \C_g \cap H} \chi(z).
    \]
    Consider now an element $x \in S_n$, 
    and apply the above formula
    with $G = S_n$, $H = Z_x = Z_{S_n}(x)$, and 
    $\chi = \omega^x$, the linear character on the centralizer $Z_x$ described in Definition~\ref{def:HLC_A}(a). Then
    \[
        \psi_{S_n}^x(g) 
        = \omega^x \up_{Z_x}^{S_n} (g)
        = \frac{|Z_g|}{|Z_x|} \sum_{z \in \C_g \cap Z_x} \omega^x(z) 
        \qquad (\forall g \in S_n).
    \]
    Recalling that $|Z_x| = |S_n|/|\C_x|$ and  $|Z_g| = |S_n|/|\C_g|$ completes the proof.  
\end{proof}

\subsection{The structure of a single cycle}

We want to study the $S_n$-structure of elements $z \in \C_g \cap Z_x$.
Our main initial result is Corollary~\ref{t:cycles2_A}, describing the $S_n$-structure of a single cycle of $z$.

Assume that $x \in S_n$ has cycle type $\lambda$.
Decompose
\[
    x = \prod_{i \ge 1} x_{i} \,,    
\]
where each $x_{i}$ is a product of $a_{i}$ disjoint cycles of length $i$ $(i \ge 1)$, and only finitely many of the factors are nontrivial. 
Then, by Lemma~\ref{t:centralizer_in_Sn},
\[
    Z_{S_n}(x)  
    = \bigtimes_{i \ge 1} Z_{S_{i a_{i}}}(x_{i}) 
    \cong \bigtimes_{i \ge 1}G_{i} \wr S_{a_{i}} ,
\]
where $G_i$ is the cyclic group of order $i$.

Assume that $g \in S_n$ has cycle type $\nu$, with $b_{j}$ cycles of length $j$ $(j \ge 1)$.
Let $z \in \C_g \cap Z_x$, and decompose it as
\[
    z = \prod_{i \ge 1} z_{i} \,,
\]
where $z_{i} \in Z_{S_{i a_{i}}}(x_{i}) \cong G_{i} \wr S_{a_{i}}$ $(i \ge 1)$.
Using a finer decomposition, assume that the underlying permutation of $z_{i}$, as an element of $S_{i a_{i}}$, has $m_{i,j}$ cycles of length $j$ $(j \ge 1)$.
Of course, 
\[
\sum_{j} j m_{i,j} = i a_{i}
\quad \text{ and } \quad 
\sum_{i} m_{i,j} = b_{j}.
\]
Assume first that, as an element of $G_{i} \wr S_{a_{i}}$, $z_{i}$ has a single cycle $c$ of length $\ell$ and $G_{i}$-class $\gamma$.
What is the structure of $c$ as an element of $S_{i a_{i}}$?

Let $\zeta_i \in \CC$ be a primitive complex $i$-th root of $1$, generating the cyclic group $G_i$ interpreted as the group of complex $i$-th roots of $1$.

\begin{lemma}\label{t:cycles_from_GwrS_to_A}
    Fix $i \ge 1$, and let $x_{i} \in S_{i a_{i}}$ be a product of $a_{i}$ cycles of length $i$.
    Let $c \in Z_{S_{i a_{i}}}(x_{i}) 
    \cong G_{i} \wr S_{a_{i}}$ be a single cycle of length $\ell$ {\rm (}in $S_{a_{i}}${\rm )} and class $\gamma = \zeta_i^k \in G_{i}$.
    If $d := \gcd(k,i)$
    then, as an element of $S_{ia_{i}}$, $c$ is a product of $d$ disjoint cycles, each of length $j = \ell i/d$.
\end{lemma}

\begin{proof}
    Following Lemma~\ref{t:centralizer_in_Sn}, 
    and after an appropriate conjugation in $S_{i a_{i}}$,
    we can assume that we have the following scenario:
    \begin{itemize}
    
        \item 
        There is an $\ell \times i$ rectangular array $(p_{s,t})$ of distinct integers in $[i a_{i}]$ (say, $p_{s,t} = (s-1)i + t$ for all $1 \le s \le \ell$ and $1 \le t \le i$).
        Each of the $\ell$ rows corresponds to a cycle (of length $i$) of $x_{i}$. As an element of $S_{i a_{i}}$,
        \begin{align*}
            x_{i} : \quad
            p_{s,t} &\mapsto p_{s,t+1}
            \qquad (1 \le s \le \ell,\, 1 \le t \le i),
        \end{align*}
        where $p_{s,i+1}$ is interpreted as $p_{s,1}$. On the other elements of $[i a_{i}]$, $x_{i}$ acts as the identity.

        \item 
        As an element of $S_{i a_{i}}$, $c \in Z_{x_{i}}$ permutes the $\ell$ rows cyclically, and specifically acts as follows:
        if $\gamma = \zeta_i^k \in G_i$ $(0 \le k < i)$, then 
        \begin{align*}
            c : \quad
            p_{s,t} &\mapsto p_{s+1,t} 
            \qquad (1 \le s \le \ell-1,\, 1 \le t \le i), \\
            p_{\ell,t} &\mapsto p_{1,t+k} 
            \qquad (1 \le t \le i),
        \end{align*}
        with $p_{1,t+k}$ interpreted as $p_{1,t+k-i}$ if $t+k > i$.        
    \end{itemize}
    Since $c$ permutes the $\ell$ rows cyclically, each cycle of $c$ (as an element of $S_{i a_{i}}$) has length divisible by $\ell$. In fact,
    \[
        c^\ell : \quad
        p_{s,t} \mapsto p_{s,t+k} 
        \qquad (\forall s,t),
    \]
    with $p_{s,t+k}$ interpreted as $p_{s,t+k-i}$ if $t+k > i$.
    Defining $d := \gcd(k,i)$, we have $\gcd(k/d,i/d) = 1$. The smallest multiple of $k$ which is divisible by $i$ is thus $(i/d) \cdot k = i \cdot (k/d)$. It follows that all the cycles of $c$ have length $\ell i/d$, and their number is $d$. 
\end{proof}

Denoting $e := i/d$, so that $j = \ell e$, we can restate Lemma~\ref{t:cycles_from_GwrS_to_A} as follows.

\begin{corollary}\label{t:cycles2_A}
    Fix $i,j \ge 1$, and let $x_{i} \in S_{i a_{i}}$ be a product of $a_{i}$ cycles of length $i$.
    To each cycle $c$ of $z_{i} \in Z_{x_{i}} \cong G_{i} \wr S_{a_{i}}$ 
    which, as an element of $S_{ia_i}$, is a product of disjoint cycles of length $j$, 
    there corresponds a common divisor $e$ of $i$ and $j$ such that, denoting $d := i/e$ and $\ell := j/e$, the following holds:
    As an element of $G_i \wr S_{a_i}$, 
    the cycle $c$ has length $\ell$ {\rm (}in $S_{a_{i}}${\rm )} and class $\gamma = \zeta_{i}^k \in G_{i}$, for some integer $k$ satisfying $\gcd(k,i) = d$; 
    and it corresponds, as an element of $S_{i \ell} = S_{j d}$, to a product of $d$ disjoint cycles of length $j$.
\end{corollary}

\subsection{Summation on a single cycle}

We now want to 
compute the sum in Lemma~\ref{t:HLC_values} on a certain small subset of $\C_g \cap Z_x$, when $g$ and $x$ have special cycle types, with all cycles of the same length.
The main result here is Lemma~\ref{t:summation_1_A}, addressing summation over elements with a single underlying cycle.

As a computational tool, recall the following well-known fact regarding the classical {\em M\"obius function} $\mu : \NN \to \{-1,0,1\}$, defined by
$\mu(1) := 1$, $\mu(n) := (-1)^k$ if $n$ is a product of $k \ge 1$ distinct primes, and $\mu(n) := 0$ otherwise (namely, if $n$ is not square-free).

\begin{lemma}\label{t:Mobius}
    For any positive integer $n$,
    \[
        \sum_{\substack{0 \le k < n \\ \gcd(k,n) = 1}} \zeta_n^k = \mu(n) \, .
    \]    
\end{lemma}

The following result performs the summation in Lemma~\ref{t:HLC_values} only on the elements of $C_g \cap Z_x$ corresponding to one specific cycle in $S_{a_{i}}$; both $g$ and $x$ are assumed to have cycle types with all cycles of the same length.

\begin{lemma}\label{t:summation_1_A}
    Fix $i,j \ge 1$, and let $x = x_{i} \in S_{i a_{i}}$ be a product of $a_{i}$ cycles of length $i$.
    Let $e$ be a common divisor of $i$ and $j$, and denote $d := i/e$ and $\ell := j/e$.
    Let $\sigma \in S_{a_{i}}$ be a permutation which has a single cycle of length $\ell$, and is the identity outside the support of this cycle.
    Let $R_{i,j}(\sigma,e)$ be the set of all the elements $z \in Z_x \cong G_{i} \wr S_{a_{i}}$ corresponding to the underlying permutation (cycle) $\sigma$, with suitable $G_{i}$-classes, such that, as elements of $S_{i \ell} = S_{j d}$, they are products of $d$ disjoint cycles of length $j$. 
    Denote
    \[
        K_{i,j}(\sigma,e) 
        := \frac{1}{i^{\ell - 1}} \sum_{z \in R_{i,j}(\sigma,e)} \omega^x(z) \, .
    \]    
    Then $K_{i,j}(\sigma,e)$ actually depends only on $e$:
    \[
        K_{i,j}(\sigma.e) = \mu(e),
    \]
    where $\mu(\cdot)$ is the classical M\"obius function.

\end{lemma}

\begin{proof}
    By 
    Corollary~\ref{t:cycles2_A},
    the set of possible $G_{i}$-classes $\gamma$ of elements $z \in R_{i,j}(\sigma,e)$ depends on $i$ and $d = i/e$, but not on $j$ or $\sigma$ (as long as $j$ is a multiple of $e$).
    Denote this set by $C(i,d)$:
    \[
       C(i,d) =
        \{\zeta_{i}^k \,:\, \gcd(k,i) = d\}. 
    \]
    The number of elements $z \in R_{i,j}(\sigma,e)$ with any specific $G_{i}$-class is $|G_{i}|^{\ell - 1} = i^{\ell - 1}$.
    Denote 
    \[
        K_{i,j}(\sigma,e) 
        := \frac{1}{i^{\ell - 1}} \sum_{z \in R_{i,j}(\sigma,e)} \omega^x(z) \, .
    \]
    It follows, by Definition~\ref{def:HLC_A}(a) of $\omega^x$, that
    \[
        K_{i,j}(\sigma,e)
        = \sum\limits_{\gamma \in C(i,d)} \gamma        = \sum_{\substack{0 \le k < i \\ \gcd(k,i) = d}} \zeta_i^k.
    \]
    Denoting $k' := k/d$ and using Lemma~\ref{t:Mobius}, we conclude that
    \[
        \sum_{\substack{0 \le k < i \\ \gcd(k,i) = d}} \zeta_i^k
        = \sum_{\substack{0 \le k' < i/d \\ \gcd(k',i/d) = 1}} \zeta_{i/d}^{k'}
        = \mu(i/d) \, ,
    \]
    namely
    \[
        K_{i,j}(\sigma,e) = \mu(e),
    \] 
    as claimed.
\end{proof}

\subsection{Proof of Theorem~\ref{t:summation_4_A}}

Extending the previous result, we shall now sum $\omega^x$ on the whole set $C_g \cap Z_x$, with increasing generality of the cycle types of $g$ and $x$. 
This will lead to a proof of Theorem~\ref{t:summation_4_A}, providing a generating function for the values of higher Lie characters $\psi_{S_n}^\lambda = \psi_{S_n}^x$.

\begin{definition}
    For positive integers $i$ and $j$, let
    \[
        E(i,j) := \{e \ge 1 \,:\, e \text{ divides both } i \text{ and } j\}.
    \]
\end{definition}

Note that $E(i,j)$ is never empty, since it always contains $e = 1$.

For an indeterminate $s$, let $\CC[[s]]$ be the ring of formal power series in $s$ over the field $\CC$.
We now extend Lemma~\ref{t:summation_1_A}, and compute the sum in Lemma~\ref{t:HLC_values} on the whole set $C_g \cap Z_x$, still restricting $g$ and $x$ to have cycle types with all cycles of the same length.

\begin{lemma}\label{t:summation_2_A}
    Fix $i,j \ge 1$.
    For any integer $m \ge 0$, let $x = x_{i}(m) \in S_{im}$ be a product of $m$ disjoint cycles of length $i$.
    Let $R_{i,j}(m)$ be the set of all elements $z \in Z_{S_{im}}(x_{i}(m)) \cong G_{i} \wr S_m$ which, as elements of $S_{im}$, are products of $im/j$ disjoint cycles of length $j$.
    (Of course, necessarily $j$ divides $im$.)
    Then, in $\CC[[s]]$,
    \[
        \sum_{m \ge 0} 
        \sum_{z \in R_{i,j}(m)} \omega^x(z)  
        \frac{s^m}{m!}
        = \exp \left( \sum_{e \in E(i,j)}
        \mu(e) \,
        \frac{(is)^{j/e}}{ij/e} \right) .
    \]
\end{lemma}

\begin{proof}    
    Assume that $E(i,j) = \{e_1, \ldots, e_q\}$, and define $\ell_k := j/e_k$ $(1 \le k \le q)$.
    For each $m \ge 0$, let 
    \[
        N_{i,j}(m) 
        := \{(n_1,\ldots,n_q) \in \ZZ_{\ge 0}^q \,:\, n_1 \ell_1 + \ldots n_q \ell_q = m\}.  
    \]
    By Corollary~\ref{t:cycles2_A}, 
    the possible cycle lengths of elements of $R_{i,j}(m)$, viewed as elements of $G_{i} \wr S_m$, are $\ell_1, \ldots, \ell_q$. If such an element has $n_k$ cycles of length $\ell_k$ $(1 \le k \le q)$, then clearly $(n_1,\ldots,n_q) \in N_{i,j}(m)$. 
    The number of permutations in $S_m$ with this cycle structure is 
    \[
        \frac{m!}{n_1! \cdots n_q! \cdot \ell_1^{n_1} \cdots \ell_q^{n_q}}.
    \]
    By Lemma~\ref{t:summation_1_A}, for each common divisor $e$ of $i$ and $j$ and each specific cycle $\sigma \in S_m$ of length $\ell = j/e$,
    \[
        \sum_{z \in R_{i,j}(\sigma,e)} \omega^x(z) 
    = i^{\ell - 1} \cdot \mu(e).
    \]    
    The linearity of the character $\omega^x$ thus implies that
    \begin{align*}
        \sum_{z \in R_{i,j}(m)} \omega^x(z) 
        &= \sum_{(n_1,\ldots,n_q) \in N_{i,j}(m)}
        \frac{m!}{\prod_{k=1}^{q} n_k! \ell_k^{n_k}}
        \cdot \prod_{k=1}^{q} 
        \left( i^{\ell_k - 1}  \mu(e_k) \right)^{n_k} \\
        &= m! \sum_{(n_1,\ldots,n_q) \in N_{i,j}(m)} 
        \prod_{k=1}^{q} \frac{1}{n_k!} \left( 
        \mu(e_k) \,
        \frac{i^{\ell_k}}{i \ell_k} \right)^{n_k} .
    \end{align*}    
    Let $s$ be an indeterminate, and consider the ring $\CC[[s]]$ of formal power series in $s$ over the field $\CC$.
    By the definition of $N_{i,j}(m)$ and the above computation, it follows that the number
    \[
        \frac{1}{m!} \sum_{z \in R_{i,j}(m)} \omega^x(z) 
    \]
    is the coefficient of $s^m$ in the product
    \begin{align*}
        \prod_{k=1}^{q} \sum_{n_k=0}^{\infty} \frac{1}{n_k!} 
        \left( \mu(e_k) \,
        \frac{(is)^{\ell_k}}{i \ell_k} \right)^{n_k}
        &= \prod_{k=1}^{q} 
        \exp \left( \mu(e_k) \,
        \frac{(is)^{\ell_k}}{i \ell_k} \right) \\
        &= \prod_{e \in E(i,j)} 
        \exp \left( \mu(e) \,
        \frac{(is)^{j/e}}{ij/e} \right) .
    \end{align*}
    In other words,
    \begin{align*}
        \sum_{m \ge 0} 
        \sum_{z \in R_{i,j}(m)} \omega^x(z) \,
        \frac{s^m}{m!} 
        &= \prod_{e \in E(i,j)}  
        \exp \left( \mu(e) \,
        \frac{(is)^{j/e}}{ij/e} \right) \\ 
        &= \exp \left( \sum_{e \in E(i,j)}  
        \mu(e) \,
        \frac{(is)^{j/e}}{ij/e} \right) . 
        \qedhere
    \end{align*}
\end{proof}

Now let $s$ be an indeterminate and $\{t_{j} \,:\, j \ge 1\}$ be a countable set of indeterminates, denoted succinctly by $\bt$. Consider the ring of formal power series $\CC[[s,\bt]]$.
We extend Lemma~\ref{t:summation_2_A} and compute the sum in Lemma~\ref{t:HLC_values} on the whole set $C_g \cap Z_x$, restricting only $x$ to have a cycle type with all cycles of the same length.
 
\begin{lemma}\label{t:summation_3_A}
    Fix $i \ge 1$.
    For any integer $m \ge 0$, let $x = x_{i}(m) \in S_{im}$ be a product of $m$ disjoint cycles of length $i$. Let $R_{i}(m) := Z_{S_{im}}(x_{i}(m)) \cong G_{i} \wr S_m$. 
    As an element of $S_{im}$, write each $z \in R_{i}(m)$ as a product of $m_{j}(z)$ disjoint cycles of length $j$ $(j \ge 1)$.
    Then, in $\CC[[s,\bt]]$,
    \[
        \sum_{m \ge 0} 
        \sum_{z \in R_{i}(m)} 
        \omega^x(z) \, 
        \frac{s^m}{m!}
        \prod_{j} t_{j}^{j m_{j}(z)}
        = \exp \left( 
        \sum_{j} \sum_{e \in E(i,j)}  
        \mu(e) \,
        \frac{(is t_{j}^i)^{j/e}}{ij/e}
        \right) .
    \]
\end{lemma}

\begin{proof}
    Following Lemma~\ref{t:summation_2_A}, fix integers $m_{j} \ge 0$ $(j \ge 1)$ such that 
    $\sum_{j} m_{j} = m$. 
    Dividing the set of $m$ cycles of $x_{i}(m)$ into subsets of corresponding sizes $m_{j}$ can be done in
    \[
        \frac{m!}{\prod_{j} m_{j}!}
    \]
    ways. On each piece $G_{i} \wr S_{m_{j}}$ we would like to consider $R_{i,j}(m_{j})$, as in Lemma~\ref{t:summation_2_A}; note that, by that result, $R_{i,j}(m_{j,\theta}) = \varnothing$ unless $j$ divides $i m_{j}$.
    By the linearity of $\omega^x$,
    \[
        \sum_{z \in R_{i}(m)} \omega^x(z)
        = \sum_{\substack{m_{j} \ge 0 \\ \sum\limits_{j} m_{j} = m}}
        \frac{m!}{\prod_{j} m_{j}!}
        \prod_{j} 
        \sum_{z_{j} \in R_{i,j}(m_{j})} \omega^x(z_{j}) 
    \]
    or, equivalently,
    \[ 
        \sum_{m \ge 0} \sum_{z \in R_{i}(m)} \omega^x(z) \,
        \frac{s^m}{m!} 
        = \prod_{j}         
        \sum_{m_{j} \ge 0}
        \sum_{z_{j} \in R_{i,j}(m_{j})} \omega^x(z_{j}) \,
        \frac{s^{m_{j}}}{m_{j}!} .
    \]
    Assume now that $z \in R_{i}(m)$, as an element of $S_{im}$, is a product of $m_{j}(z)$ disjoint cycles of length $j$ $(j \ge 1)$.
    This yields a subdivision of the set of $m$ cycles (of length $i$ each) of $x_{i}(m)$ into subsets of sizes $m_{j} = j m_{j}(z)/i$, so that
    \[
        \sum_{j} j m_{j}(z) 
        = \sum_{j} i m_{j} 
        = i m.
    \]
    In order to keep track of the individual numbers $j m_{j}(z) = i m_{j}$, let us use additional indeterminates $t_{j}$ $(j \ge 1)$. The previous formula turns into
    \[ 
        \sum_{m \ge 0} \sum_{z \in R_{i}(m)} \omega^x(z) \,
        \frac{s^m}{m!} 
        \prod_{j} t_{j}^{j m_{j}(z)}
        = \prod_{j}         
        \sum_{m_{j} \ge 0}
        \sum_{z_{j} \in R_{i,j}(m_{j})} \omega^x(z_{j}) \,
        \frac{s^{m_{j}} t_{j}^{i m_{j}}}{m_{j}!} .
    \]
    Rewriting the RHS using Lemma~\ref{t:summation_2_A}, with $s$ replaced by $s t_{j}^i$, yields
    \begin{align*}
        \sum_{m \ge 0} 
        \sum_{z \in R_{i}(m)} 
        \omega^x(z) \,
        \frac{s^m}{m!}
        \prod_{j} t_{j}^{j m_{j}(z)} 
        &= \prod_{j} \exp \left(
        \sum_{e \in E(i,j)}  
        \mu(e) \,
        \frac{(is t_{j}^i)^{j/e}}{ij/e} 
        \right) \\
        &= \exp \left( 
        \sum_{j} \sum_{e \in E(i,j)} 
        \mu(e) \,
        \frac{(is t_{j}^i)^{j/e}}{ij/e} 
        \right) ,
    \end{align*}
    as claimed.
\end{proof}

Finally, let $s_{i}$ $(i \ge 1)$ and $t_{j}$ $(j \ge 1)$ be two countable sets of indeterminates, denoted succinctly by $\bs$ and $\bt$.
Consider the ring of formal power series $\CC[[\bs,\bt]]$.
For $\lambda \vdash n$ with $a_{i}$ parts of length $i$ $(i \ge 1)$, denote
\[
    \bs^{c(\lambda)} := \prod_{i} s_{i}^{a_{i}} . 
\]
Similarly, for $\nu \vdash n$ with $b_{j}$ parts of length $j$ $(j \ge 1)$, denote 
\[
    \bt^{c(\nu)} := \prod_{j} t_{j}^{b_{j}} . 
\]
Recalling the higher Lie characters $\psi_{S_n}^\lambda$ from Definition~\ref{def:HLC_A}(c), we can now prove the main result of this section, Theorem~\ref{t:summation_4_A}.
It extends Lemma~\ref{t:summation_3_A} and computes the sum in Lemma~\ref{t:HLC_values} on the whole set $C_g \cap Z_x$, for arbitrary $g$ and $x$.

\begin{proof}[Proof of Theorem~\ref{t:summation_4_A}]
    Write any $x \in S_n$ as 
    \[
        x = \prod_{i} x_{i} \,,
    \]
    where each $x_{i}$ is a product of $a_{i}$ disjoint cycles of length $i$.
    Then, by Lemma~\ref{t:centralizer_in_Sn},
    \[
        Z_x \cong \bigtimes_{i} Z_{x_{i}}
    \]
    where
    \[
        Z_{x_{i}} \cong G_{i} \wr S_{a_{i}}.
    \]
    By Lemma~\ref{t:summation_3_A}, with summation over $m = a_{i} \ge 0$ and $z = z_{i} \in G_{i} \wr S_{a_{i}}$, we have for each $i \ge 1$:
    \[
        \sum_{a_{i} \ge 0} 
        \sum_{z_{i} \in G_{i} \wr S_{a_{i}}} \omega^{x_{i}}(z_{i}) \,
        \frac{s^{a_{i}}}{a_{i}!}
        \prod_{j} t_{j}^{j m_{j}(z_{i})}
        = \exp \left( 
        \sum_{j} \sum_{e \in E(i,j)}  
        \mu(e) \,
        \frac{(is t_{j}^i)^{j/e}}{ij/e} 
        \right) . 
    \]
    Now replace 
    $s$ by $s_{i}/i$ and $t_j^j$ by $t_j$ $(\forall j \ge 1)$. 
    Denote
    \begin{align*}
        \Sigma_{i}
        &:= \sum_{a_{i} \ge 0} 
        \sum_{z_{i} \in G_{i} \wr S_{a_{i}}} \omega^{x_{i}} (z_{i}) \,
        \frac{s_{i}^{a_{i}}}{a_{i}! \, i^{a_{i}}}
        \prod_{j} t_{j}^{m_{j}(z_{i})} \\
        &= \exp \left( 
        \sum_{j} \sum_{e \in E(i,j)} 
        \mu(e) \,
        \frac{s_{i}^{j/e} t_{j}^{i/e}}{ij/e} 
        \right) . 
    \end{align*}
    The product of $\Sigma_{i}$ over all $i$ is therefore
    \begin{align*}
        &\ \prod_{i} \sum_{a_{i} \ge 0} 
        \sum_{z_{i} \in Z_{x_{i}}} \omega^{x_{i}}(z_{i}) \, 
        \frac{s_{i}^{a_{i}}}{a_{i}! \, i^{a_{i}}}
        \prod_{j} t_{j}^{m_{j}(z_{i})} \\
        &= \prod_{i} \Sigma_{i} 
        = \exp \left( 
        \sum_{i} \sum_{j} \sum_{e \in E(i,j)}  
        \mu(e) \,
        \frac{s_{i}^{j/e} t_{j}^{i/e}}{ij/e} 
        \right) .
    \end{align*}
    If $x = \prod_{i} x_{i} \in C_\lambda$ then
    \[
        |C_\lambda| 
        = |C_x|
        = \frac{|S_n|}{|Z_x|} 
        = \frac{|S_n|}{\prod_{i} a_{i}!\, i^{a_{i}}} .
    \]
    The above equality can thus be written as
    \begin{align*}
        &\ \sum_{n \ge 0} \frac{1}{|S_n|} 
        \sum_{\lambda \vdash n} \sum_{x \in C_\lambda} 
        \prod_{i} \left( 
        \sum_{z_{i} \in Z_{x_{i}}} 
        \omega^{x_{i}}(z_{i})  
        s_{i}^{a_{i}} \prod_{j} t_{j}^{m_{j}(z_{i})} 
        \right) \\
        &= \exp \left( 
        \sum_{i} \sum_{j} \sum_{e \in E(i,j)}  
        \mu(e) \,
        \frac{s_{i}^{j/e} t_{j}^{i/e}}{ij/e} 
        \right) .
    \end{align*}
    Denote $z := \prod_{i} z_{i} \in Z_x$, and note that
    \[
        \sum_{i} m_{j}(z_{i}) = b_{j}(z)
        \qquad (\forall j \ge 1) .
    \]
    By the definition of $\omega^x$, the LHS of the equality can thus be written as
    \[
        \text{LHS} 
        = \sum_{n \ge 0} \frac{1}{|S_n|} 
        \sum_{\lambda \vdash n} \sum_{x \in C_\lambda} \sum_{z \in Z_x} 
        \omega^x(z) \bs^{c(x)} \bt^{c(z)} ,
    \]
    where 
    \[
        \bs^{c(x)} := \prod_{i} s_{i}^{a_{i}}, \quad
        \bt^{c(z)} := \prod_{j} t_{j}^{b_{j}} .
    \]
    In fact, we can rewrite this as
    \[
        \text{LHS} 
        = \sum_{n \ge 0} \frac{1}{|S_n|} 
        \sum_{\lambda \vdash n} \sum_{\nu \vdash n} \sum_{x \in C_\lambda} \sum_{z \in Z_x \cap C_\nu} 
        \omega^x(z) 
        \bs^{c(\lambda)} \bt^{c(\nu)} ,
    \]
    since $c(x)$ depends only on the conjugacy class of $x$, and may thus be written as $c(\lambda)$; and similarly for $c(z)$ and $c(\nu)$.
    
    Now, by Lemma~\ref{t:HLC_values}, if $x \in C_\lambda$ then
    \[
        |C_\lambda| \cdot \sum_{z \in Z_x \cap C_\nu} \omega^x(z) 
        = |C_\nu| \cdot \psi_{S_n}^x(\nu) .
    \]
    Therefore
    \begin{align*}
        \text{LHS} 
        &= \sum_{n \ge 0} \frac{1}{|S_n|} 
        \sum_{\lambda \vdash n} \sum_{\nu \vdash n} 
        |C_\nu| \, \psi_{S_n}^\lambda(\nu) \,
        \bs^{c(\lambda)} \bt^{c(\nu)} \\
        &= \sum_{n \ge 0}  
        \sum_{\lambda \vdash n} \sum_{\nu \vdash n} 
        \psi_{S_n}^\lambda(\nu) \,
        \frac{\bs^{c(\lambda)} \bt^{c(\nu)}}{|Z_\nu|} .
    \end{align*}
    Regarding the RHS, note that $E(i,j)$ is the set of all common divisors of $i$ and $j$, namely divisors of $\gcd(i,j)$. We can therefore write the above equality as 
    \begin{align*}
        \sum_{n \ge 0}  
        \sum_{\lambda \vdash n} \sum_{\nu \vdash n} 
        \psi_{S_n}^\lambda(\nu) \,
        \frac{\bs^{c(\lambda)} \bt^{c(\nu)}}{|Z_\nu|} 
        &= \exp \left( 
        \sum_{i} \sum_{j} \sum_{e | \gcd(i,j)}  
        \mu(e) \,
        \frac{s_{i}^{j/e} t_{j}^{i/e}}{ij/e} 
        \right) .
    \end{align*}
    This completes the proof.
\end{proof}

\section{A signed version}
\label{sec:gf_hLc_A_signed}

In this section we outline the proof of the following signed analogue of Theorem~\ref{t:summation_4_A}.

\begin{theorem}\label{t:summation_4_A_signed}
    \begin{align*}
        \sum_{n \ge 0}  
        \sum_{\lambda \vdash n} \sum_{\nu \vdash n} 
        \sign(\nu) \psi_{S_n}^\lambda(\nu) \,
        \frac{\bs^{c(\lambda)} \bt^{c(\nu)}}{|Z_\nu|} 
        &= \exp \left( \sum_{i} \sum_{j} \sum_{e | \gcd(i,j)}
        (-1)^{i(j-1)/e} \mu(e) \,
        \frac{s_{i}^{j/e} t_{j}^{i/e}}{ij/e} 
        \right) ,
    \end{align*}
    where $\mu(\cdot)$ is the classical M\"obius function.
\end{theorem}

The proof is very similar to the proof of Theorem~\ref{t:summation_4_A}, described in Section~\ref{sec:gf_hLc_A}. 
We state only the slightly different main lemmas, without proof.

Here is a signed analogue of Lemma~\ref{t:summation_1_A}.

\begin{lemma}\label{t:summation_1_A_signed}
    Fix $i,j \ge 1$, and let $x = x_{i} \in S_{i a_{i}}$ be a product of $a_{i}$ cycles of length $i$.
    Let $e$ be a common divisor of $i$ and $j$, and denote $d := i/e$ and $\ell := j/e$.
    Let $\sigma \in S_{a_{i}}$ be a permutation which has a single cycle of length $\ell$, and is the identity outside the support of this cycle.
    Let $R_{i,j}(\sigma,e)$ be the set of all the elements $z \in Z_x \cong G_{i} \wr S_{a_{i}}$ corresponding to the underlying permutation (cycle) $\sigma$, with suitable $G_{i}$-classes, such that, as elements of $S_{i \ell} = S_{j d}$, they are products of $d$ disjoint cycles of length $j$. 
    Denote
    \[
        K_{i,j}^-(\sigma,e) 
        := \frac{1}{i^{\ell - 1}} \sum_{z \in R_{i,j}(\sigma,e)} \sign(z) \omega^x(z) \, .
    \]    
    Then
    \[
        K_{i,j}^-(\sigma.e) = (-1)^{i(j-1)/e} \mu(e).
    \]
\end{lemma}

Here is a signed analogue of Lemma~\ref{t:summation_2_A}.

\begin{lemma}\label{t:summation_2_A_signed}
    Fix $i,j \ge 1$.
    For any integer $m \ge 0$, let $x = x_{i}(m) \in S_{im}$ be a product of $m$ disjoint cycles of length $i$.
    Let $R_{i,j}(m)$ be the set of all elements $z \in Z_{S_{im}}(x_{i}(m)) \cong G_{i} \wr S_m$ which, as elements of $S_{im}$, are products of $im/j$ disjoint cycles of length $j$.
    (Of course, necessarily $j$ divides $im$.)
    Then, in $\CC[[s]]$,
    \[
        \sum_{m \ge 0} 
        \sum_{z \in R_{i,j}(m)} \sign(z) \omega^x(z) 
        \frac{s^m}{m!}
        = \exp \left( \sum_{e \in E(i,j)}
        (-1)^{i(j-1)/e} \mu(e) \,
        \frac{(is)^{j/e}}{ij/e} \right) .
    \]
\end{lemma}

Here is a signed analogue of Lemma~\ref{t:summation_3_A}.

\begin{lemma}\label{t:summation_3_A_signed}
    Fix $i \ge 1$.
    For any integer $m \ge 0$, let $x = x_{i}(m) \in S_{im}$ be a product of $m$ disjoint cycles of length $i$. Let $R_{i}(m) := Z_{S_{im}(x_{i}(m))} \cong G_{i} \wr S_m$. 
    As an element of $S_{im}$, write each $z \in R_{i}(m)$ as a product of $m_{j}(z)$ disjoint cycles of length $j$ $(j \ge 1)$.
    Then, in $\CC[[s,\bt]]$,
    \[
        \sum_{m \ge 0} 
        \sum_{z \in R_{i}(m)} 
        \sign(z) \omega^x(z) \, 
        \frac{s^m}{m!}
        \prod_{j} t_{j}^{j m_{j}(z)}
        = \exp \left( 
        \sum_{j} \sum_{e \in E(i,j)}  
        (-1)^{i(j-1)/e} \mu(e) \,
        \frac{(is t_{j}^i)^{j/e}}{ij/e} 
        \right) .
    \]
\end{lemma}

\section{Proofs of Theorems~\ref{t:odd_cycles} and~\ref{t:even_cycles}}
\label{sec:proofs}



\begin{proof}[Proof of Theorem~\ref{t:odd_cycles}]
    Use Theorem~\ref{t:summation_4_A} with the substitution
    \[
        s_i :=
        \begin{cases}
            1, &\text{if $i$ is odd;} \\
            0, &\text{if $i$ is even}
        \end{cases}
    \]
    to get
    \begin{align*}
        \sum_{n \ge 0}  
        \sum_{\lambda \in \OP(n)} 
        \sum_{\nu \vdash n} 
        \psi_{S_n}^\lambda(\nu) \,
        \frac{\bt^{c(\nu)}}{|Z_\nu|} 
        &= \exp \left( \sum_{i \text{ odd}} \sum_{j \ge 1} \sum_{e | \gcd(i,j)}
        \mu(e) \,
        \frac{t_{j}^{i/e}}{ij/e} \right) .
    \end{align*}
    Replacing (in the exponent) summation over $i$ (odd) and divisors $e$ by summation over $e$ (odd) and $d := i/e$ (odd),
    it follows that
    \begin{align*}
        \sum_{n \ge 0}  
        \sum_{\lambda \in \OP(n)} 
        \sum_{\nu \vdash n} 
        \psi_{S_n}^\lambda(\nu) \,
        \frac{\bt^{c(\nu)}}{|Z_\nu|} 
        &= \exp \left( \sum_{j \ge 1} 
        \sum_{\substack{e \,|\, j \\ e \text{ odd}}} \mu(e) \,
        \sum_{d \text{ odd}}
        \frac{t_{j}^{d}}{dj} \right) .
    \end{align*}
    Writing $j = 2^p \cdot (2q+1)$ for integers $p, q \ge 0$, it is clear that
    \[
        \sum_{\substack{e \,|\, j \\ e \text{ odd}}} \mu(e) 
        = \sum_{e \,|\, (2q+1)} \mu(e) 
        = \delta_{q,0}.
    \]
    Therefore necessarily $j = 2^p$, and using
    \[
        \sum_{d \text{ odd}} \frac{x^{d}}{d} 
        = \frac{1}{2} \left( \ln(1+x) - \ln(1-x) \right) 
    \]
    it follows that
    \begin{align*}
        \sum_{n \ge 0}  
        \sum_{\lambda \in \OP(n)} 
        \sum_{\nu \vdash n} 
        \psi_{S_n}^\lambda(\nu) \,
        \frac{\bt^{c(\nu)}}{|Z_\nu|} 
        &= \exp \left( \sum_{p \ge 0} \sum_{d \text{ odd}}
        \frac{t_{2^p}^{d}}{2^p d} \right) 
        = \prod_{p \ge 0} 
        \left( \frac{1 + t_{2^p}}{1 - t_{2^p}} \right)^{1/2^{p+1}} .
    \end{align*}
\end{proof}




\begin{proof}[Proof of Theorem~\ref{t:even_cycles}]
    We shall deal separately with summation over even $n$ and summation over odd $n$.
    For even $n$, since $\EP(n)$ consists of the partitions of $n$ with even parts only,
    we can use Theorem~\ref{t:summation_4_A_signed} with the substitution
    \[
        s_i :=
        \begin{cases}
            1, &\text{if $i$ is even;} \\
            0, &\text{if $i$ is odd}
        \end{cases}
    \]
    to get
    \begin{align*}
        \sum_{n \ge 0 \text{ even}}  
        \sum_{\lambda \in \EP(n)} 
        \sum_{\nu \vdash n} 
        \sign(\nu) \psi_{S_n}^\lambda(\nu) \,
        \frac{
        \bt^{c(\nu)}}{|Z_\nu|} 
        &= \exp \left( 
        \sum_{i \text{ even}} \sum_{j \ge 1} \sum_{e | \gcd(i,j)}
        (-1)^{i(j-1)/e} \mu(e) \,
        \frac{t_{j}^{i/e}}{ij/e} \right) .
    \end{align*}
    Replace double summation over $i$ (even) and divisors $e$ by double summation over $e$ and $d := i/e$.
    Distinguishing the cases of $e$ odd (thus $d$ even) and $e$ even (thus $d$ of arbitrary parity), it follows that
    \begin{align*}
        \sum_{n \ge 0 \text{ even}}  
        \sum_{\lambda \in \EP(n)} 
        \sum_{\nu \vdash n} 
        \sign(\nu) \psi_{S_n}^\lambda(\nu) \,
        \frac{
        \bt^{c(\nu)}}{|Z_\nu|} 
        &= \exp \left( 
        \sum_{j \ge 1} 
        \sum_{\substack{e \,|\, j \\ e \text{ odd}}} \mu(e) \,
        \sum_{d \text{ even}}
        \frac{t_{j}^{d}}{dj} \quad + \right. \\
        &\qquad \quad \; \left.
        \sum_{j \text{ even}} 
        \sum_{\substack{e \,|\, j \\ e \text{ even}}} \mu(e) \,
        \sum_{d \ge 1}
        (-1)^{d(j-1)}
        \frac{t_{j}^{d}}{dj} \right) .
    \end{align*}
    Writing $j = 2^p \cdot (2q+1)$ for integers $p, q \ge 0$, it is clear that
    \[
        \sum_{\substack{e \,|\, j \\ e \text{ odd}}} \mu(e) 
        = \sum_{e \,|\, (2q+1)} \mu(e) 
        = \delta_{q,0}
    \]
    and (for $j$ even, namely $p \ge 1$)
    \[
        \sum_{\substack{e \,|\, j \\ e \text{ even}}} \mu(e) 
        = \sum_{\substack{e_1 \,|\, 2^p \\ e_1 \ne 1}} \mu(e_1) 
        \cdot \sum_{e_2 \,|\, (2q+1)} \mu(e_2) 
        = (-1) \cdot \delta_{q,0} .
    \]
    Therefore, in both cases, necessarily $j = 2^p$.
    Noting that $2^p - 1$ is odd for $p \ge 1$, we deduce that
    \begin{align*}
        \sum_{n \ge 0 \text{ even}}  
        \sum_{\lambda \in \EP(n)} 
        \sum_{\nu \vdash n} 
        \sign(\nu) \psi_{S_n}^\lambda(\nu) \,
        \frac{
        \bt^{c(\nu)}}{|Z_\nu|} 
        &= \exp \left( 
        \sum_{p \ge 0} 
        \sum_{d \text{ even}}
        \frac{t_{2^p}^{d}}{2^p d} 
        - 
        \sum_{p \ge 1} 
        \sum_{d \ge 1}
        (-1)^{d(2^p-1)}
        \frac{t_{2^p}^{d}}{2^p d} \right) \\
        &= \exp \left( 
        \sum_{d \text{ even}}
        \frac{t_{1}^{d}}{d}
        + 
        \sum_{p \ge 1} 
        \sum_{d \text{ even}}
        \frac{t_{2^p}^{d}}{2^p d} 
        - 
        \sum_{p \ge 1} 
        \sum_{d \ge 1}
        (-1)^d \frac{t_{2^p}^{d}}{2^p d} \right) \\
        &= \exp \left( 
        \sum_{d \text{ even}}
        \frac{t_{1}^{d}}{d}
        + 
        \sum_{p \ge 1} 
        \sum_{d \text{ odd}}
        \frac{t_{2^p}^{d}}{2^p d} \right) .
    \end{align*}
    Since 
    \[
        \exp \left( 
        \sum_{d \text{ even}} \frac{x^{d}}{d} 
        \right) 
        = 
        \exp \left( -\frac{1}{2} \ln \left( 1-x^2 \right) \right)
        = 
        \left( 1-x^2 \right)^{-1/2} ,
        \]
    and
    \[
        \exp \left(
        \sum_{d \text{ odd}} \frac{x^{d}}{d}
        \right)
        = \exp \left(
        \frac{1}{2} \ln(1+x) - \frac{1}{2} \ln(1-x)
        \right)
        = \left( \frac{1+x}{1-x} \right)^{1/2} ,
    \]
    it follows that
    \begin{align*}
        \sum_{n \ge 0 \text{ even}}  
        \sum_{\lambda \in \EP(n)} 
        \sum_{\nu \vdash n} 
        \sign(\nu) \psi_{S_n}^\lambda(\nu) \,
        \frac{
        \bt^{c(\nu)}}{|Z_\nu|} 
        &= 
        \left( 1-t_1^2 \right)^{-1/2} \cdot
        \prod_{p \ge 1} 
        \left( \frac{1 + t_{2^p}}{1 - t_{2^p}} \right)^{1/2^{p+1}} .
    \end{align*}
    Now consider the summation over odd $n$.
    Since $\EP(n)$, in this case, consists of the partitions of $n$ with even parts only, except for a single part of size $1$,
    we can use Theorem~\ref{t:summation_4_A_signed} with the substitution
    \[
        s_i :=
        \begin{cases}
            1, &\text{if $i$ is even;} \\
            s_1, &\text{if $i = 1$;} \\
            0, &\text{if $i > 1$ is odd}
        \end{cases}
    \]
    to get that
    \[
        \sum_{n \text{ odd}}  
        \sum_{\lambda \in \EP(n)} 
        \sum_{\nu \vdash n} 
        \sign(\nu) \psi_{S_n}^\lambda(\nu) \,
        \frac{s_1 \cdot
        \bt^{c(\nu)}}{|Z_\nu|} 
    \]
    is the sum of all monomials containing a single $s_1$ in the expansion of
    \[
        \exp \left( 
        \sum_{j \ge 1} \sum_{e | \gcd(1,j)}
        (-1)^{(j-1)/e} \mu(e) \,
        \frac{s_1^{j/e} t_{j}^{1/e}}{j/e} 
        +
        \sum_{i \text{ even}} \sum_{j \ge 1} \sum_{e | \gcd(i,j)}
        (-1)^{i(j-1)/e} \mu(e) \,
        \frac{t_{j}^{i/e}}{ij/e} \right) .
    \]
    In the first summation, corresponding to $i = 1$, clearly $e = 1$ and therefore (in order to get a single $s_1$) also $j = 1$.
    The only monomial containing a single $s_1$ in the expansion of $\exp (s_1 t_1)$ is $s_1 t_1$.
    Therefore, using the computation for even $n$,
    \begin{align*}
        \sum_{n \text{ odd}}  
        \sum_{\lambda \in \EP(n)} 
        \sum_{\nu \vdash n} 
        \sign(\nu) \psi_{S_n}^\lambda(\nu) \,
        \frac{
        \bt^{c(\nu)}}{|Z_\nu|} 
        &= t_1 \cdot
        \left( 1-t_1^2 \right)^{-1/2} \cdot
        \prod_{p \ge 1} 
        \left( \frac{1 + t_{2^p}}{1 - t_{2^p}} \right)^{1/2^{p+1}} .
    \end{align*}
    Adding together the summations over even and odd $n$, and using
    \[
        (1 + t_1) \cdot (1 - t_1^2)^{-1/2}
        = \left(
        \frac{1 + t_1}{1 - t_1}
        \right)^{1/2} ,
    \]
    completes the proof.    
\end{proof}

\section{Root enumerators}\label{sec:roots}

\subsection{The odd root enumerator}
\label{sec:odd_roots}

Recall that, for an integer $k$ and a nonnegative integer $n$, the $k$-th root enumerator in $S_n$ is defined by
\[
    \roots_k^{S_n}(g) 
    := |\{x \in S_n \,:\, x^k = g\}| 
    = \sum_{\substack{x \in S_n \\ x^k = g}} 1 
    \qquad (\forall g \in S_n) .
\]
For a partition $\lambda$ of $n$, denote $\lambda \vdash_k n$ if all the part sizes of $\lambda$ divide $k$. 
The following theorem, connecting root enumerators with higher Lie characters, is due to Scharf. 

\begin{theorem}\label{t:roots_eq_psi_Sn}\cite{Scharf}
    For any integer $k$ and nonnegative integer $n$,
    \[
        \roots^{S_n}_k
        = \sum_{\lambda \vdash_k n} \psi_{S_n}^\lambda.
    \]
\end{theorem}

Denote
\[
\roots_{odd}^{S_{n}}(g) := |\{x \in S_n \,:\, x^k = g\ \text{ for some odd } k\}|
    \qquad (\forall g \in S_n).
\]

\begin{corollary}\label{cor:Scharf_odd}
    For any positive integer $n$,
    \[
        \roots_{odd}^{S_{n}}
        = \sum_{\lambda \in \OP(n)}
        \psi^\lambda_{S_{n}} .
    \]
\end{corollary}

\begin{proof}
    Let $k$ be any odd multiple of $lcm \{2i-1 \,:\, 1 \le i \le (n+1)/2\}$. Clearly,
    $\roots_{odd}^{S_n} = \roots_k^{S_n}$. 
    Scharf's theorem (Theorem~\ref{t:roots_eq_psi_Sn}) completes the proof.
\end{proof}

We deduce

\begin{proposition}\label{t:induced_to_odd}
    For any positive integer $n$,
    \[
        \roots_{odd}^{S_{2n+1}} 
        = \roots_{odd}^{S_{2n}} \up_{S_{2n}}^{S_{2n+1}}.
    \]
\end{proposition}

\begin{proof}
    By Corollary~\ref{cor:Scharf_odd} together with Theorem~\ref{t:OP-EP}, for any positive integer $n$,
    \[
        \roots_{odd}^{S_{2n+1}}
        = \sum_{\lambda \in \OP(2n+1)} 
        \psi_{S_{2n+1}}^\lambda
        = \sign \otimes \sum_{\lambda \in \EP(2n+1)} 
        \psi_{S_{2n+1}}^\lambda ,
    \]
    and similarly
    \[
        \roots_{odd}^{S_{2n}}
        = \sum_{\lambda \in \OP(2n)} 
        \psi_{S_{2n}}^\lambda
        = \sign \otimes \sum_{\lambda \in \EP(2n)} 
        \psi_{S_{2n}}^\lambda .
    \]
    By the definition of higher Lie characters, for any partition $\lambda = (\lambda_1, \ldots, \lambda_t) \vdash 2n$ with no parts of size 1,
    \[
        \psi^{(\lambda_1, \ldots, \lambda_t,1)}_{S_{2n+1}}
        = \psi^{(\lambda_1, \ldots, \lambda_t)}_{S_{2n}} \up_{S_{2n}}^{S_{2n+1}}.
    \]
    It follows that
    \begin{align*}
        \roots_{odd}^{S_{2n+1}}
        &= \sign \otimes \sum_{\lambda \in \EP(2n+1)}
        \psi_{S_{2n+1}}^\lambda
        = \sign \otimes \sum_{\lambda \in \EP(2n)} 
        \psi_{S_{2n}}^\lambda \up_{S_{2n}}^{S_{2n+1}} \\
        &= \left( \sign \otimes \sum_{\lambda \in \EP(2n)} 
        \psi_{S_{2n}}^\lambda \right)\up_{S_{2n}}^{S_{2n+1}}
        = \roots_{odd}^{S_{2n}} \up_{S_{2n}}^{S_{2n+1}}. \qedhere
    \end{align*}
\end{proof}

%

\begin{remark}
    Note that, by Corollary~\ref{cor:double},   
    for every $n$, $\roots_{odd}^{2n}$ is not induced from $\roots_{odd}^{2n-1}$. 
\end{remark}

\subsection{The signed odd root enumerator}
\label{sec:sined_roots}

Let $\sign: S_n \to \{1,-1\}$ be the sign character on $S_n$.

\begin{definition}\label{def:sroots_Bn}
    Let $k$ be an integer and $n$ a nonnegative integer.
    The {\em signed $k$-th root enumerator} in $S_n$ is defined by 
    \[
        \sroots_k^{S_n}(g) 
        := \sum_{\substack{x \in S_n \\ x^k = g}} \sign(x)
        \qquad (g \in S_n) .
    \]
\end{definition}

\begin{remark}
    The following generating function
    was obtained by Lea\~nos, Moreno and Rivera-Mart\'inez~\cite{Leanos}.
    In our notation, with the obvious modifications, their result is:
    \[
        \sum_{n \ge 0} 
        \sum_{\nu \vdash n} 
        \roots_k^{S_n}(\nu) \,
        \frac{\bt^{c(\nu)}}{|Z_\nu|}
        = \exp \left(
        \sum_{j \ge 1}
        \sum_{\substack{h|k \\ \gcd(h,j) = 1}} 
        \frac{t_j^{k/h}}{j k/h}
        \right) .
    \]
    Formally, they use $g = k/h$ instead of our $h$, and sum over all $g \ge 1$ such that $\gcd(gj,k) = g$; see Remark 1 after the proof of~\cite[Proposition 2]{Leanos}. Their $t_j$ is our $t_j/j$.

    If $k$ ia a prime power, this generating function appears already in Wilf's book  Generatingfunctionology, Theorem 4.8.3.
\end{remark}

\begin{remark}
    The corresponding expression for the sign character of the symmetric group was obtained by Chernoff~\cite{Chernoff} and by Glebsky, Lic\'on and Rivera~\cite{GLR}.
    In our notation, including the class function $\sroots_k^{S_n}(g) := \sum_{x \in S_n,\, x^k = g} \sign(x)$ (for $g \in S_n$), namely $\sroots_k^{S_n}(\nu)$ (for a cycle type $\nu \vdash n$),
    their result is:
    \[
        \sum_{n \ge 0} 
        \sum_{\nu \vdash n} 
        \sroots_k^{S_n}(\nu) \,
        \frac{\bt^{c(\nu)}}{|Z_\nu|}
        = \exp \left(
        \sum_{j \ge 1}
        \sum_{\substack{h|k \\ \gcd(h,j) = 1}} 
        \frac{(-1)^{1 + jk/h} t_j^{k/h}}{j k/h}
        \right) .
    \]
\end{remark}

\begin{definition} The {\em signed odd root enumerator} is defined as 
\[
    \sroots_{odd}^{S_{n}}(g) 
    := \sum_{\substack{x \in S_n \\ x^k = g \ \text{for some odd } k}} \sign (x).
    \qquad (\forall g \in S_n) .
\]
\end{definition}


\begin{definition}\label{def:signed_HLC_A}
    {\rm (Twisted higher Lie characters in $S_n$)}
    Let $x$ be an element of cycle type $\lambda$ in $S_n$, as in 
    Definition~\ref{def:HLC_A}.
    
    \begin{enumerate}
    
        \item[(a)]
        For each $i \ge 1$, let $\oomega_{i}$ be the linear character on $G_{i} \wr S_{a_{i}}$ which is equal to a primitive irreducible character on the cyclic group $G_{i}$,
        and equal to the {\em sign character} on the wreathing group $S_{a_{i}}$.
        Let
        \[
            \oomega^x 
            := 
            \bigotimes_{i \ge 1} \oomega_{i} ,
        \]
        a linear character on $Z_x$. 
        
        \item[(b)]
        Define the corresponding {\em twisted higher Lie character} to be the induced character
        \[
            \tau_{S_n}^x 
            := \oomega^x 
            \up_{Z_x}^{S_n}. 
        \]

        \item[(c)]
        It is easy to see that $\tau_{S_n}^x$ depends only on the conjugacy class $C$ (equivalently, the cycle type $\lambda$) of $x$, and can therefore be denoted $\tau_{S_n}^C$ or $\tau_{S_n}^\lambda$.
        
    \end{enumerate}
\end{definition}

The following theorem is due to Scharf.

\begin{theorem}\label{t:twist_Scharf}~\cite[2.4.22~Satz]{Scharf_thesis}
For every partition $\lambda\vdash n$,
\[
        \sroots_{k}^{S_{n}}
        = \sum_{\lambda \vdash_k n} 
        \sign(\lambda) \tau^\lambda_{S_{n}},    
\]
where $\tau^\lambda_{S_{n}}$ is the twisted higher Lie character from Definition~\ref{def:signed_HLC_A} 
and $\sign(\lambda)$ is the sign of a permutation of cycle type $\lambda$.
\end{theorem}

Note that $\sign(\lambda) = 1$ for any $\lambda \in \OP(n)$.

\begin{corollary}\label{t:twist}
    For any positive integer $n$,
    \[
        \sroots_{odd}^{S_{n}}
        = \sum_{\lambda \in \OP(n)} 
        \tau^\lambda_{S_{n}}. 
    \]
\end{corollary}


We deduce the following identity. 

\begin{theorem}
        For any positive integer $n$,
    \[
         \sum_{\lambda \in \OP(n)}
         \tau^\lambda_{S_{n}}
         = \sum_{\lambda \in \EP(n)} 
         \psi_{S_n}^\lambda, 
    \]
    where $\tau^\lambda_{S_{n}}$ is the twisted higher Lie character from Definition~\ref{def:signed_HLC_A} and $\psi^\lambda_{S_n}$ is the standard higher Lie character.
\end{theorem}

\begin{proof}
    By Theorem~\ref{t:OP-EP}, 
    Corollary~\ref{cor:Scharf_odd},
    Corollary~\ref{t:twist}, and the fact that the sign of an odd power of $x$ is equal to the sign of $x$:
    \[
          \sum_{\lambda \in \EP(n)} 
          \psi_{S_n}^\lambda
          = \sign \otimes \sum_{\lambda \in \OP(n)} 
          \psi_{S_n}^\lambda
        = \sign \otimes \, \roots_{odd}^{S_n}
        = \sroots_{odd}^{S_n}
        = \sum_{\lambda \in \OP(n)} 
        \tau^\lambda_{S_{n}}. 
        \qedhere
    \]
   \end{proof}

\section{Induction from a subgroup of type \texorpdfstring{$B$}{B}}
\label{sec:induction}


The 
group of signed permutations
$B_n = \ZZ_2 \wr S_n$ may be viewed as the centralizer, in $S_{2n}$, of a fixed-point-free involution (a permutation of cycle type $(2,\ldots,2)$).
The fact that its index $|S_{2n}|/|B_n| = (2n-1)!!$ divides the degree $(2n-1)!!^2$ of the odd roots enumerator $\roots_{odd}^{S_{2n}}$
raises a natural question: 
Is the character $\roots_{odd}^{S_{2n}}$ induced from a character of $B_n$?

In this section we answer this question affirmatively. We show that $\roots_{odd}^{S_{2n}}$ is induced from the sum of higher Lie characters of type $B$ corresponding to conjugacy classes in $B_n$ with positive cycles only. The number of signed permutations in $B_n$ with these cycle types is indeed $(2n-1)!!$, see Proposition~\ref{prop:Bn_positive} below. Thus the degree of the induced representation is indeed $(2n-1)!!^2$, 
which is the number of odd roots of the identity permutation in $S_{2n}$, in accordance with Corollary~\ref{cor:double}.  
Theorem~\ref{t:induced_to_even} shows that, furthermore, the two characters are equal.

\subsection{Signed permutations with positive cycles only}
\label{subsec:conjugacy_B}



Denote $[\pm n] := \{-n, \ldots, -1, 1, \ldots, n\}$.
The elements of $B_n = \ZZ_2 \wr S_n$ are {\em signed permutations}, namely bijections $\sigma:[\pm n] \to [\pm n]$ which satisfy $\sigma(-i)=-\sigma(i)$ for all $i$. 
They are encoded by the window (i.e., sequence of values on $1, \ldots, n$) $[\sigma(1),\ldots,\sigma(n)]$.  

Conjugacy classes in $B_n=\ZZ_2\wr S_n$ are parametrized by pairs of partitions $(\lambda^+,\lambda^-)$ of total size $n$.
A cycle in a signed permutation $\sigma \in B_n$ which contains a letter $i$ as well as $-i$ is called {\em negative}; otherwise it is called {\em positive}.  
The conjugacy class $C_{(\lambda^+,\lambda^-)}$ consists of all the signed permutations whose positive cycle lengths are the parts of $\lambda^+$ and the negative cycle lengths, divided by 2, are the parts of $\lambda^-$. 

\begin{proposition}\label{prop:Bn_positive}
    The number of signed permutations in $B_n$ with positive cycles only is $(2n-1)!!$.
\end{proposition}

\begin{proof}
    Consider the set $M_{2n}$ of perfect matchings on $2n$ points, labeled by the letters in $[\pm n]$. Let $m_0 \in M_{2n}$ match $i$ with $-i$, for all $1\le i\le n$. The superposition of $m_0$ and any perfect matching $m\in M_{2n}$ is 
    a union of vertex-disjoint cycles of even lengths, with edges alternating between $m_0$ and $m$. 
    Choose an orientation for each of the cycles; for concreteness, pick in each cycle the largest positive letter, say $i_0$, and orient the cycle so that the edge of $m_0$ connecting $i_0$ and $-i_0$ is oriented away from $i_0$. 
    Define a function $\sigma_m : [\pm n] \to [\pm n]$ as follows: locate each element $v \in [\pm n]$ in a cycle, and follow the path of two consecutive edges starting at $v$, an edge from $m$ followed by an edge from $m_0$ (or vice versa), according to the orientation of the cycle, to get $\sigma_m(v) \in [\pm n]$. The reader may verify that $\sigma_m \in B_n$, has only positive cycles, and $m \mapsto \sigma_m$ defines a bijection from $M_{2n}$ onto the set of all permutations in $B_n$ with positive cycles. The fact that $|M_{2n}| = (2n-1)!!$ completes the proof. 
\end{proof}

\subsection{Higher Lie characters of type \texorpdfstring{$B$}{B}}
\label{subsec:def_hLc_B}

Recall from~\cite{AHR}  
the definition of higher Lie characters of type $B$, parametrized by conjugacy classes in $B_n \cong \ZZ_2 \wr S_n$.

\begin{lemma}\label{t:centralizer_in_Bn}
    {\rm (Centralizers in $B_n$)}
    \begin{enumerate}
    
        \item[(a)]
        Let $x \in B_n$ be an element of cycle type $\blambda = (\lambda^+,\lambda^-)$ where, for each $i \ge 1$ and $\ve \in \ZZ_2$, the partition $\lambda^\ve$ has $a_{i,\ve}$ parts of size $i$. 
        Write 
        \[
            x = \prod_{i \ge 1} x_{i,+} x_{i,-}\,, 
        \]
        where $x_{i,\ve}$ has $a_{i,\ve}$ cycles of length $i$ and $\ZZ_2$-class $\ve$ $(i \ge 1,\, \ve \in \ZZ_2)$; of course, only finitely many factors here are nontrivial.
        Then the centralizer 
        $Z_x = Z_{B_n}(x)$ satisfies
        \[
            Z_{B_n}(x) = \bigtimes_{i \ge 1} \left( Z_{B_{ia_{i,+}}}(x_{i,+}) \times Z_{B_{ia_{i,-}}}(x_{i,-}) \right) ,
        \]
        and is therefore isomorphic to the direct product 
        \[
            \bigtimes_{i \ge 1} \left( G_{i,+} \wr S_{a_{i,+}} \times G_{i,-} \wr S_{a_{i,-}} \right) ,
        \]
        where $G_{i,\ve}$ is the centralizer in $B_i$ of a cycle of length $i$ and $\ZZ_2$-class $\ve$.
        By convention, $G \wr S_0$ is the trivial group while $G \wr S_1 \cong G$. 
        
        \item[(b)]
        Let $x_{i,+} \in B_i$ consist of a single cycle, of length $i$ and class $+1 \in \ZZ_2$. Then the centralizer $ Z_{x_{i,+}} = Z_{B_i}(x_{i,+}) = G_{i,+}$ is isomorphic to the group $\ZZ_2 \times \ZZ_i$, where the generator of $\ZZ_i$ is $x_{i,+}$ and the generator of $\ZZ_2$ is the central {\rm (}longest{\rm )} element $w_0 = [-1,\ldots,-i] \in B_i$.
        \item[(c)]
        Let $x_{i,-} \in B_i$ consist of a single cycle, of length $i$ and class $-1 \in \ZZ_2$. Then the centralizer $Z_{x_{i,-}} = Z_{B_i}(x_{i,-}) = G_{i,-}$ is isomorphic to the cyclic group $\ZZ_{2i}$, with generator $x_{i,-}$; note that $x_{i,-}^i = w_0$.
        
    \end{enumerate}
\end{lemma}

\begin{definition}\label{def:HLC_Bn}
    {\rm (Higher Lie characters in $B_n$)}
    Let $x$ be an element of cycle type $\blambda = (\lambda^+,\lambda^-)$ in $B_n$, as in Lemma~\ref{t:centralizer_in_Bn}(a).        
    \begin{enumerate}
    
        \item[(a)]
        For each $i \ge 1$ and $\ve \in \ZZ_2$, let $\omega_{i,\ve}$ be the linear character on $G_{i,\ve} \wr S_{a_{i,\ve}}$ defined as follows: 
        If $\ve = +1$ then, by Lemma~\ref{t:centralizer_in_Bn}(b), $G_{i,+} \cong \ZZ_2 \times \ZZ_i$.
        Let $\omega_{i,+}$ be trivial on $\ZZ_2$, 
        equal to a primitive irreducible character on the cyclic group $\ZZ_i$,
        and trivial on the wreathing group $S_{a_{i,+}}$.
        If $\ve = -1$ then, by Lemma~\ref{t:centralizer_in_Bn}(c), $G_{i.-} \cong \ZZ_{2i}$.
        Let $\omega_{i,-}$ be equal to a primitive irreducible character on the cyclic group $\ZZ_{2i}$,
        and trivial on the wreathing group $S_{a_{i,-}}$.
        Let
        \[
            \omega^x 
            := 
            \bigotimes_{i \ge 1} \left( \omega_{i,+} \otimes \omega_{i,-} \right),
        \]
        a linear character on $Z_x$. 
        
        \item[(b)]
        Define the corresponding {\em higher Lie character} to be the induced character
        \[
            \psi_{B_n}^x 
            := \omega^x 
            \uparrow_{Z_x}^{B_n}. 
        \]

        \item[(c)]
        It is easy to see that $\psi_{B_n}^x$ depends only on the conjugacy class $C$ (equivalently, the cycle type $\blambda$) of $x$, and can therefore be denoted $\psi_{B_n}^C$ or $\psi_{B_n}^\blambda$.
        
    \end{enumerate}
\end{definition}

\subsection{The odd root enumerator as an induced character}

The centralizer of any permutation of cycle type $(2^n)$ (i.e., a fixed-point-free involution) in $S_{2n}$ is isomorphic to $B_n$.
Similarly, the centralizer of any permutation of cycle type $(2^n 1)$ in $S_{2n+1}$ is isomorphic to $B_n$.
Denote
\[
    \eta_{B_n}
    := \sum_{\lambda \vdash n} \psi_{B_n}^{(\lambda, \varnothing)} ,
\]
the sum of type $B$ higher Lie characters corresponding to conjugacy classes with positive cycles only.

\begin{theorem}\label{t:induced_to_even}
    For any $n \ge 0$,
    \[
        \roots_{odd}^{S_{2n}} 
        = \eta_{B_{n}} \up_{B_n}^{S_{2n}}
    \]
    and
    \[
        \roots_{odd}^{S_{2n+1}} 
        = \eta_{B_{n}} \up_{B_n}^{S_{2n+1}} .
    \]
\end{theorem}

For the proof of Theorem~\ref{t:induced_to_even}, we first need the following general lemma, which is of independent interest.

\begin{lemma}\label{t:induced_character}
    Let $G$ be a finite group, $H$ a subgroup of $G$, $\chi$ a character of $H$, and $\xi = \chi \up_H^G$ the corresponding induced character of $G$. 
    Let $\{\C_G(\ell) \,:\, \ell \in L\}$ and $\{\C_H(m) \,:\, m \in M\}$ be the sets of conjugacy classes in $G$ and $H$, respectively, where $L$ and $M$ are appropriate indexing sets.
    Let $|Z_G(\ell)| = |G|/|C_G(\ell)|$ and $|Z_H(m)| = |H|/|C_H(m)|$ be the corresponding sizes of centralizer subgroups. 
    Let $\alpha : M \to L$ be the function defined by 
    \[
        \C_H(m) \subseteq C_G(\alpha(m))
        \qquad (\forall m \in M).
    \]
    Finally, let $\{q_\ell \,:\, \ell \in L\}$ be a set of indeterminates.
    Then
    \[
        \sum_{\ell \in L} \frac{\xi(\ell) \cdot q_\ell}{|Z_G(\ell)|}
        = \sum_{m \in M} \frac{\chi(m) \cdot q_{\alpha(m)}}{|Z_H(m)|}. 
    \]
\end{lemma}

\begin{proof}
    Define a function $\chi^0: G \to \CC$ by
    \[
        \chi^0(g) :=
        \begin{cases}
            \chi(g), &\text{if } g \in H; \\
            0, &\text{if } g \in G \setminus H.
        \end{cases}
    \]
    By~\cite[(5.1)]{Isaacs}, an explicit formula for the induced character $\xi = \chi \up_H^G$ is
    \[
        \xi(g)
        = \sum_{G = \cup_a aH} \chi^0(a^{-1} g a)
        = \frac{1}{|H|} \sum_{a\in G} \chi^0(a^{-1} g a)
        \qquad (\forall g \in G) .
    \]
    The mapping $f : G \to \C_G(g)$ defined by
    $f(a) := a^{-1} g a$ $(\forall a \in G)$ is surjective, and satisfies:
    $f(a_1) = f(a_2)$ if and only if $a_1 a_2^{-1} \in Z_G(g)$.
    Hence
    \[
        \xi(g)
        = \frac{|Z_G(g)|}{|H|} \sum_{z \in \C_G(g)}\chi^0(z)
        = \frac{|Z_G(g)|}{|H|} \sum_{h \in \C_G(g) \cap H} \chi(h)
        \qquad (\forall g \in G).
    \]
    Assume that the $G$-conjugacy class of $g$ is $\C_G(\ell)$.
    The intersection $\C_G(\ell) \cap H$ is a disjoint union of the $H$-conjugacy classes $\C_H(m)$ for which $\alpha(m) = \ell$.
    Summing over all $\ell \in L$, it follows that
    \[
        \sum_{\ell \in L} 
        \frac{\xi(\ell) \cdot q_\ell}{|Z_G(\ell)|}
        = \sum_{\ell \in L} 
        \frac{q_\ell}{|H|}
        \sum_{h \in C_G(\ell) \cap H} \chi(h)
        = \sum_{m \in M} \frac{q_{\alpha(m)}}{|H|}
        \sum_{h \in C_H(m)} \chi(h)
        = \sum_{m \in M} \frac{q_{\alpha(m)} \cdot |C_H(m)|}{|H|}
        \chi(m),
    \]
    which is equivalent to the required formula.  
\end{proof}

The proof of Theorem~\ref{t:induced_to_even} is also based on a generating function for higher Lie characters of the hyperoctahedral group $B_n$, proved in~\cite{AHR}.

Let $\blambda = (\lambda^+,\lambda^-)$ and $\bnu = (\nu^+,\nu^-)$ be two bipartitions of $n$. 
For each integer $i \ge 1$ and sign $\ve \in \ZZ_2 = \{+1,-1\}$, let $a_{i,\ve}$ be the number of parts of size $i$ in the partition $\lambda^\ve$.
Similarly, for each integer $j \ge 1$ and sign $\theta \in \ZZ_2 = \{+1,-1\}$, let $b_{j,\theta}$ be the number of parts of size $j$ in the partition $\nu^\theta$. 
Thus
\[
    \sum_{i,\ve} i a_{i,\ve} 
    = \sum_{j,\theta} j b_{j,\theta} 
    = n.
\]
Let $\bs = (s_{i,\ve})_{i \ge 1, \ve \in \ZZ_2}$ and $\bt = (t_{j,\theta})_{j \ge 1, \theta \in \ZZ_2}$ be two countable sets of indeterminates.
Denote 
$\bs^{c(\blambda)} := \prod_{i,\ve} s_{i,\ve}^{a_{i,\ve}}$ 
and $\bt^{c(\bnu)} := \prod_{j,\theta} t_{j,\theta}^{b_{j,\theta}}$.
Consider the ring $\CC[[\bs,\bt]]$ of formal power series in these indeterminates.

\begin{remark}
    The notation used in~\cite{AHR} is slightly different from the notation used in the current paper: wherever we use $s_{i,\ve}$ and $t_{j,\theta}$, \cite{AHR} uses $s_{i,\ve}^i$ and $t_{j,\theta}^j$, respectively. 
    The following result is stated in our current notation.
\end{remark}

\begin{theorem}\label{t:summation_4_Bn}
{\rm \cite[Theorem 4.3]{AHR}}
    \begin{align*}
        \sum_{n \ge 0}  
        \sum_{\blambda \vdash n} \sum_{\bnu \vdash n} 
        \psi_{B_n}^\blambda(\bnu) \,
        \frac{\bs^{c(\blambda)} \bt^{c(\bnu)}}{|Z_{B_n}(\bnu)|} 
        &= \exp \left( \sum_{i,\ve} \sum_{j,\theta} \sum_{e | \gcd(i,j)}
        K_{\ve,\theta}(e) \,
        \frac{s_{i,\ve}^{j/e} t_{j,\theta}^{i/e}}{2ij/e}
        \right) .
    \end{align*}
    Here, for $\ve, \theta \in \ZZ_2 = \{+1,-1\}$ and $e \ge 1$,
    \[
        K_{\ve,\theta}(e)
        := \ve \theta \cdot \mu(2e)
        + \frac{(1+\ve)(1+\theta)}{2} \cdot \mu(e),
    \] 
    where $\mu(\cdot)$ is the classical M\"obius function.
\end{theorem}

\begin{proof}[Proof of Theorem~\ref{t:induced_to_even}]
    Following the definition of $\eta_{B_n}$ set, in Theorem~\ref{t:summation_4_Bn},
    \[
        s_{i,\ve} 
        := \begin{cases}
            1, &\text{if $\ve = +1$;} \\
            0, &\text{otherwise}
        \end{cases}
    \]
    for all $i \ge 1$. Then
    \begin{align*}
        \sum_{n \ge 0}  
        \sum_{\bnu \vdash n} 
        \eta_{B_n}(\bnu) \,
        \frac{\bt^{c(\bnu)}}{|Z_{B_n}(\bnu)|} 
        &= \exp \left( \sum_{i} \sum_{j,\theta} \sum_{e | \gcd(i,j)}
        K_{+,\theta}(e) \,
        \frac{t_{j,\theta}^{i/e}}{2ij/e} 
        \right) .
    \end{align*}
    Inducing from $B_n$ to $S_{2n}$ (or to $S_{2n+1}$), a cycle of length $j$ and sign $+1$ in $B_n$ is a product of two disjoint cycles of length $j$ in $S_n$, while a cycle of length $j$ and sign $-1$ in $B_n$ is a cycle of length $2j$ in $S_n$. We therefore set
    \[
        \alpha(t_{j,+}) = t_j^2
        \qquad (\forall j \ge 1)
    \]
    and
    \[
        \alpha(t_{j,-}) = t_{2j}
        \qquad (\forall j \ge 1).
    \]
    Consider first induction to $S_{2n}$.
    By Lemma~\ref{t:induced_character} for $G = S_{2n}$, $H = B_n$, $\chi = \sum_{\bnu \vdash n} \eta_{B_n}(\bnu)$ and $\xi_{2n} = \chi \up_{B_n}^{S_{2n}}$,
    \begin{align*}
        \sum_{n \ge 0}  
        \sum_{\nu \vdash 2n} 
        \xi_{2n}(\nu) \,
        \frac{\bt^{c(\nu)}}{|Z_{S_{2n}}(\nu)|} 
        &= \exp \left( \sum_{i} \sum_{j} \sum_{e | \gcd(i,j)}
        \frac{K_{+,+}(e) \cdot t_j^{2i/e} + K_{+,-}(e) \cdot t_{2j}^{i/e}}{2ij/e} 
        \right) .
    \end{align*}
    Using
    \[  
        K_{+,+}(e) = \mu(2e) + 2 \cdot \mu(e)
    \]
    and
    \[
        K_{+,-}(e) = - \mu(2e),
    \]
    and denoting $d := i/e$, we get
    \begin{align*}
        \sum_{n \ge 0}  
        \sum_{\nu \vdash 2n} 
        \xi_{2n}(\nu) \,
        \frac{\bt^{c(\nu)}}{|Z_{S_{2n}}(\nu)|} 
        &= \exp \left( 
        \sum_{d \ge 1} \sum_{j \ge 1} \sum_{e | j}
        \frac{(\mu(2e) + 2 \cdot \mu(e)) \cdot t_j^{2d} - \mu(2e) \cdot t_{2j}^{d}}{2dj} 
        \right) .
    \end{align*}
    Since
    \[
        \sum_{e | j} \mu(2e)
        = - \sum_{\substack{e | j \\ e \text{ odd}}} \mu(e)
        = \begin{cases}
            -1, &\text{if $j = 2^p$ for $p \ge 0$;} \\
            0, &\text{otherwise}
        \end{cases}
    \]
    and
    \[
        \sum_{e | j} (\mu(2e) + 2 \cdot \mu(e))
        = \begin{cases}
            1, &\text{if } j = 1; \\
            -1, &\text{if $j = 2^p$ for $p \ge 1$;} \\
            0, &\text{otherwise,}
        \end{cases}
    \]
    it follows that
    \begin{align*}
        \sum_{n \ge 0}  
        \sum_{\nu \vdash 2n} 
        \xi_{2n}(\nu) \,
        \frac{\bt^{c(\nu)}}{|Z_{S_{2n}}(\nu)|} 
        &= \exp \left( 
        \sum_{d \ge 1} 
        \frac{t_1^{2d} + t_2^{d}}{2d}         
        + \sum_{d \ge 1} \sum_{p \ge 1} 
        \frac{- t_{2^p}^{2d} + t_{2^{p+1}}^{d}}{2^{p+1} d} 
        \right) \\
        &= \exp \left( 
        \sum_{d \ge 1} 
        \frac{t_1^{2d}}{2d}         
        + 
        \sum_{d \ge 1} \sum_{p \ge 1} 
        \frac{-t_{2^p}^{2d}}{2^{p+1} d}        
        + 
        \sum_{d \ge 1} 
        \frac{t_2^{d}}{2d}         
        + 
        \sum_{d \ge 1} \sum_{p \ge 1} 
        \frac{t_{2^{p+1}}^{d}}{2^{p+1} d} 
        \right) \\
        &= \exp \left( 
        \sum_{d \ge 1} 
        \frac{t_1^{2d}}{2d}         
        + 
        \sum_{d \ge 1} \sum_{p \ge 1} 
        \frac{-t_{2^p}^{2d}}{2^{p+1} d}        
        + 
        \sum_{d \ge 1} \sum_{p \ge 1} 
        \frac{t_{2^{p}}^{d}}{2^{p} d} 
        \right) \\
        &= (1 - t_1^2)^{-1/2}
        \cdot \prod_{p \ge 1} \left(
        1 - t_{2^p}^2 \right)^{1/2^{p+1}} 
        \cdot \prod_{p \ge 1} \left(
        1 - t_{2^p} \right)^{-1/2^{p}} \\
        &= (1 - t_1^2)^{-1/2}
        \cdot \prod_{p \ge 1} \left(
        \frac{1 + t_{2^p}}{1 - t_{2^p}} 
        \right)^{1/2^{p+1}} .
    \end{align*}
    A similar application of Lemma~\ref{t:induced_character} with $G = S_{2n+1}$, $H = B_n$, $\chi = \sum_{\bnu \vdash n} \eta_{B_n}(\bnu)$ and $\xi_{2n+1} = \chi \up_{B_n}^{S_{2n+1}}$ yields
    \[
        \sum_{n \ge 0}  
        \sum_{\nu \vdash 2n+1} 
        \xi_{2n+1}(\nu) \,
        \frac{\bt^{c(\nu)}}{|Z_{S_{2n+1}}(\nu)|} 
        = t_1 \cdot (1 - t_1^2)^{-1/2}
        \cdot \prod_{p \ge 1} \left(
        \frac{1 + t_{2^p}}{1 - t_{2^p}} 
        \right)^{1/2^{p+1}} ,
    \]
    so altogether we have
    \[
        \sum_{n \ge 0}  
        \sum_{\nu \vdash n} 
        \xi_{n}(\nu) \,
        \frac{\bt^{c(\nu)}}{|Z_{S_{n}}(\nu)|} 
        = \prod_{p \ge 0} \left(
        \frac{1 + t_{2^p}}{1 - t_{2^p}} 
        \right)^{1/2^{p+1}} .
    \]
    Comparing this formula to the one in Theorem~\ref{t:odd_cycles} and to Corollary~\ref{cor:Scharf_odd} shows that, indeed,
    \[
        \xi_{n}
        = \sum_{\lambda \in \OP(n)} 
        \psi_{S_{n}}^\lambda 
        = \roots_{odd}^{S_n}
    \]
    for all $n \ge 0$, as claimed.
\end{proof}


\begin{thebibliography}{MM}

\bibitem{AHR}
    R.\ M.\ Adin, P.\ Heged\"{u}s and Y.\ Roichman,
    {\em Higher Lie characters and root enumeration in classical Weyl groups},
    J.\ Algebra~{\bf 664} (2025), 26--72.

\bibitem{Bona} 
    M.\ B\'ona, 
    A walk through combinatorics, 
    World Scientific Publishing Co.\ Pte.\ Ltd., Hackensack, NJ, 2006.

\bibitem{BMW}
M.\ B\'ona, S.\ McLennan and D.\ White, {\em Permutations with roots}, Random Structures Algorithms 17 (2000), 157-–167.

\bibitem{Chernoff}
    W.\ W.\ Chernoff, 
    {\em Solutions to $x^r = \alpha$ in the alternating group}, 
    Twelfth British Combin.\ Conf.\ (Norwich 1989), Ars Combin.~{\bf 29(C)} (1990), 226--227.

\bibitem{GR}
    I.\ M.\ Gessel and C.\ Reutenauer, 
    {\em Counting permutations with	given cycle structure and descent set}, 
    J.\ Combin.\ Theory Series A~{\bf 64} (1993), 189--215.	

\bibitem{GLR}
    L.\ Glebsky, M.\ Lic\'on and L.\ M.\ Rivera
    {\em On the number of even roots of permutations}, 
    Australasian J.\ Comb.~{\bf 86} (2023), 308--319.

\bibitem{Isaacs}
    I.\ M.\ Isaacs,
    {\em Character Theory of Finite Groups},
    Dover, New York, 1994.

\bibitem{Leanos}
    J.\ Lea\~nos, R.\ Moreno and L.\ M.\ Rivera-Mart\'inez,
    {\em On the number of $m$th roots of permutations},
    Australasian J.\ Comb.~{\bf 52} (2012), 41--54.

\bibitem{Scharf} 
    T.\ Scharf,
    {\em Die Wurzelanzahlfunktion in symmetrischen Gruppen},
    J.\ Algebra~{\bf 139}, No.\ 2 (1991), 446--457. 

\bibitem{Scharf_thesis} 
    \bysame, 
    \"Uber Wurzelanzahlfunktionen voller monomialer Gruppen (German) [On root number functions of full monomial groups],
    Dissertation, Universit\"at Bayreuth, Bayreuth, 1991.
    Bayreuth.\ Math.\ Schr.~{\bf 38} (1991), 99--207.

\bibitem{EC2}
    R.\ P.\ Stanley, 
    Enumerative Combinatorics, Vol.\ 2, 
    Cambridge Studies in Adv.\ Math.\ 62, Cambridge Univ.\ Press, Cambridge, 1999.

\end{thebibliography}
\end{document}